\documentclass[a4paper, 12pt, reqno]{amsart}
\usepackage{amsfonts, amsmath, amssymb, amsthm}
\usepackage[english]{babel}
\usepackage[colorlinks=true, linkcolor=blue, citecolor=blue]{hyperref}
\usepackage[utf8]{inputenc}
\usepackage{tikz-cd}
\usepackage[outline]{contour}
\contourlength{1.2pt}
\advance\hoffset-20mm \advance\textwidth40mm
\renewcommand{\AA}{\mathbb{A}}
\newcommand{\FF}{\mathbb{F}}
\newcommand{\GG}{\mathbb{G}}
\newcommand{\KK}{\mathbb{K}}
\newcommand{\M}{\mathfrak{M}}

\newcommand{\PP}{\mathbb{P}}
\newcommand{\PGL}{\mathrm{PGL}}
\newcommand{\QQ}{\mathbb{Q}}
\newcommand{\R}{\mathfrak{R}}
\renewcommand{\S}{\mathfrak{S}}
\newcommand{\T}{\mathfrak{T}}
\newcommand{\U}{\mathfrak{U}}
\renewcommand{\u}{\mathfrak{u}}
\newcommand{\Umax}{U_\mathrm{max}}
\newcommand{\Uss}{U_\mathrm{ss}}
\newcommand{\ZZ}{\mathbb{Z}}
\newcommand{\Aut}{\mathrm{Aut}}
\newcommand{\Cl}{\mathrm{Cl}}
\newcommand{\Cone}{\mathrm{Cone}}
\renewcommand{\div}{\mathrm{div}}
\newcommand{\GL}{\mathrm{GL}}
\newcommand{\Hom}{\mathrm{Hom}}
\newcommand{\id}{\mathrm{id}}
\newcommand{\Lie}{\mathrm{Lie}}
\newcommand{\Spec}{\mathrm{Spec}}
\newcommand{\WDiv}{\mathrm{WDiv}}
\newcommand{\closure}[1]{\overline{#1}}
\renewcommand{\hat}[1]{\widehat{#1}}
\newcommand{\pairing}[2]{\langle #1, #2 \rangle}

\renewcommand{\succeq}{\succcurlyeq}
\newtheorem{section_numbering_1}{}[section]
\theoremstyle{definition}
\newtheorem{Algorithm}[section_numbering_1]{Algorithm}
\newtheorem{Example}[section_numbering_1]{Example}
\newtheorem{Remark}[section_numbering_1]{Remark}
\newtheorem{definition}[section_numbering_1]{Definition}
\newtheorem{Problem}[section_numbering_1]{Problem}
\theoremstyle{plain}

\newtheorem{Corollary}[section_numbering_1]{Corollary}
\newtheorem{Lemma}[section_numbering_1]{Lemma}
\newtheorem{lemma}[section_numbering_1]{Lemma}
\newtheorem{Proposition}[section_numbering_1]{Proposition}
\newtheorem{Theorem}[section_numbering_1]{Theorem}
\newtheorem*{TheoremA}{Theorem A}
\begin{document}
\sloppy 
\title[Radiant toric varieties and unipotent group actions]{Radiant toric varieties and unipotent group actions}
\author{Ivan Arzhantsev}
\address{HSE University, Faculty of Computer Science, Pokrovsky Boulevard 11, Moscow, 109028 Russia}
\email{arjantsev@hse.ru}
\author{Alexander Perepechko}
\address{Kharkevich Institute for Information Transmission Problems,  19 Bolshoy Karetny per., Moscow, 127994 Russia}
\address{HSE University, Faculty of Computer Science, Pokrovsky Boulevard 11, Moscow, 109028 Russia}
\email{a@perep.ru}
\author{Kirill Shakhmatov}
\address{HSE University, Faculty of Computer Science, Pokrovsky Boulevard 11, Moscow, 109028 Russia}
\email{bagoga@list.ru}
\dedicatory{To Jürgen Hausen on the occasion of his 60th birthday}
\thanks{Supported by the grant RSF-DST 22-41-02019}
\begin{abstract}
We consider complete toric varieties $X$ such that a maximal unipotent subgroup~$U$ of the automorphism group $\Aut(X)$ acts on $X$ with an open orbit. It turns out that such varieties can be characterized by several remarkable properties. We study the set of Demazure roots of the corresponding complete fan, describe the structure of a maximal unipotent subgroup $U$ in $\Aut(X)$, and find all regular subgroups in $U$ that act on $X$ with an~open orbit. 
\end{abstract} 
\maketitle

%%%%%%%%%%%%%%%%%%%%%%%%%%%%%%%%%%%%%%%%

\section*{Introduction}
Let $X$ be a complete toric variety with an acting torus $T$. All varieties and actions are defined over an algebraically closed field $\KK$ of characteristic zero. It is proved in \cite[Proposition~11]{Dem}, see also \cite[Theorem~4.2]{Cox} and \cite[Theorem~4.2.4.1]{ADHL}, that the automorphism group $\Aut(X)$ is a linear algebraic group and $T$ is a maximal torus in $\Aut(X)$. An algebraic subgroup $H$ in $\Aut(X)$ is called \emph{regular} if $H$ is normalized by the torus $T$.

We say that $X$ is \emph{radiant} if a maximal unipotent subgroup $U$ of the automorphism group $\Aut(X)$ acts on $X$ with an open orbit. The aim of this paper is to show that radiant toric varieties have important distinguishing features that make this class of varieties attractive for study.

Let $\GG_a$ and $\GG_m$ stand for the additive and the multiplicative group of the ground field~$\KK$, respectively. The groups $\GG_a$ and $\GG_m$ are commutative one-dimensional linear algebraic groups, and $\GG_a$ is unipotent. Moreover, any commutative unipotent linear algebraic group is isomorphic to the vector group $\GG_a^n=\GG_a\times\ldots\times\GG_a$ ($n$~times). Given an algebraic action $\GG_a\times X\to X$ (resp. $\GG_m\times X\to X$), we call the image of $\GG_a$ (resp. $\GG_m$) in $\Aut(X)$ a \emph{$\GG_a$-subgroup} (resp. \emph{$\GG_m$-subgroup}) in $\Aut(X)$. An \emph{additive action} on a complete algebraic variety $X$ is a faithful algebraic action $\GG_a^n\times X\to X$ with an open orbit. Since the group $\GG_a^n$ is isomorphic to the affine space $\AA^n$ as an algebraic variety, we obtain an open embedding of $\AA^n$ into $X$ such that the action of $\GG_a^n$ on $\AA^n$ by translations extends to an action on $X$. This shows that a complete variety with an additive action is an equivariant completion of the affine space, and vice versa; see~\cite{HT,AZ}. 

If $X$ is toric we say that an additive action $\GG_a^n\times X\to X$ is \emph{normalized} if the image of the group $\GG_a^n$ in $\Aut(X)$ is a regular subgroup. Let us say that a regular 
$\GG_a$-subgroup in $\Aut(X)$ is a \emph{root} subgroup. It is easy to see that a normalized additive action is an additive action given by $n$ pairwise commuting root subgroups in $\Aut(X)$.

It is well known that any complete toric variety $X$ corresponds to a complete fan $\Sigma=\Sigma_X$ in the lattice $N$ of one-parameter subgroups of the acting torus $T$. Let 
$\rho_1,\ldots,\rho_m$ be the rays of the fan $\Sigma$. We say that the fan $\Sigma$ is \emph{bilateral} if there is a basis $p_1,\ldots,p_n$ of the lattice $N$ such that, up to renumbering,
the rays $\rho_1,\ldots,\rho_n$ are generated by the vectors $p_1,\ldots,p_n$, respectively, and the remaining rays $\rho_{n+1},\ldots,\rho_m$ lie in the negative orthant with respect to this basis.

The following result is proved in \cite[Theorem~4.1 and Corollary~1]{AR}.

\begin{TheoremA}
Let $X$ be a complete toric variety. The following conditions are equivalent:
\begin{enumerate}
\item[(i)]
the variety $X$ is radiant;
\item[(ii)]
the variety $X$ admits an additive action;
\item[(iii)]
the variety $X$ admits a normalized additive action;
\item[(iv)]
the fan $\Sigma_X$ is bilateral.
\end{enumerate}
\end{TheoremA}

In particular, radiant toric varieties are precisely toric equivariant completions of affine spaces, cf.~\cite{AZ}.

One more characterization of radiant projective toric varieties was proposed by Fu and Hwang~\cite{FH}. Let $X\subseteq\PP(V)$ be a closed subvariety. For a nonsingular point $x\in X$, a $\GG_m$-action on $X$ coming from a $\GG_m$-subgroup of $\GL(V)$ is said to be of \emph{Euler type} at~$x$ if $x$ is an isolated fixed point of the restricted $\GG_m$-action on $X$ and the induced $\GG_m$-action on the tangent space $T_xX$ is by scalar multiplication. The subvariety $X\subseteq\PP(V)$ is \emph{Euler-symmetric} if for a general point $x\in X$, there exists a $\GG_m$-action on $X$ of Euler type at $x$. It is proved in~\cite[Theorem~3.7(i)]{FH} that  every Euler-symmetric variety admits an additive action. If $X$ is a projective toric variety, it is shown in~\cite[Theorem~3]{Sh} that $X$ admits an additive action if and only if $X$ is Euler-symmetric with respect to some embedding into a projective space. Moreover, in this case $X$ is Euler-symmetric with respect to any linearly non-degenerate linearly normal embedding into a projective space.

\smallskip

Lattice polytopes that correspond to projective radiant toric varieties are described in~\cite[Theorem~5.2]{AR}.
Examples of non-projective radiant toric varieties are given in~\cite{Sha}; see also \mbox{Example~\ref{exsha}} below. 

\smallskip

It is shown in \cite[Theorem~3.6]{AR} that a normalized additive action on a radiant toric variety is unique up to isomorphism. Moreover, it follows from~\cite[Corollary~4]{D2} that any additive action on a radiant toric variety is isomorphic to a normalized additive action if and only if the maximal unipotent subgroup $U$ in $\Aut(X)$ is commutative. In this case $U$ is the only candidate up to conjugation for a subgroup in $\Aut(X)$ defining an additive action on $X$.

\smallskip

Let us describe the content of the paper. In Section~\ref{sec1} we discuss basic facts on toric varieties and Cox rings and recall a description of the automorphism group of a complete toric variety due to Demazure. Section~\ref{sec2} is devoted to the study of Demazure roots of a bilateral fan. Unlike classical root systems, the set of Demazure roots of a complete polyhedral fan may seem rather random and unsymmetric. However, a more detailed analysis shows that in the case of a bilateral fan, the set of Demazure roots has many nice properties. The study of such properties was started by Dzhunusov~\cite{D1, D2}. We develop these results and find new features of Demazure roots. With any bilateral fan $\Sigma$ we associate its ray matrix $A$. We show that Demazure roots of $\Sigma$ admit an explicit description in terms of columns of the matrix $A$. This allows both to study properties of Demazure roots and to construct bilateral fans with prescribed Demazure roots.

In Section~\ref{sec3}, using results on Demazure roots, we give a description of a maximal unipotent subgroup of the automorphism group $\Aut(X)$ of a radiant toric variety $X$ as a semidirect product of some explicitly presented unipotent groups. The next step is to find all regular unipotent subgroups in $\Aut(X)$ which act on $X$ with an open orbit. This is done in Section~\ref{sec4}. We define a \emph{principal unipotent subgroup} of $\Aut(X)$ as a regular commutative unipotent subgroup that acts on $X$ with an open orbit. It is proved that a regular unipotent subgroup $U$ in $\Aut(X)$ acts on $X$ with an open orbit if and only if $U$ contains a principal unipotent subgroup. We come to consider two types of radiant toric varieties: the ones that admit a faithful action of a non-commutative unipotent group and the ones that do not. Specific properties of radiant toric varieties of each type are studied.  

In Section~\ref{sec5} we compute the center of a regular unipotent subgroup $U$ acting on $X$ with an open orbit. With any such $U$ we associate a directed graph $\Gamma$ and describe the central series and the derived subgroups of $U$ in terms of $\Gamma$. These results allow to find the nilpotency class and the derived length of a regular unipotent subgroup. 

It is shown in Section~\ref{sec6} that the open orbit of a non-commutative regular unipotent subgroup has smaller dimension than the subgroup itself. This imposes certain restrictions on equivariant toric completions of unipotent groups. Finally, in Section~\ref{sec7} we describe smooth radiant toric surfaces. It turns out that the Picard number of a smooth non-radiant toric surface is at least~4. We expect further classification results on smooth radiant toric varieties of small dimensions.

\smallskip

The authors thank the referee for a careful reading and helpful comments.

%%%%%%%%%%%%%%%%%%%%%%%%%%%%%%%%%%%%%%%%

\section{Preliminaries}
\label{sec1} 

In this section we recall basic facts on toric varieties. Hereinafter we fix an algebraically closed field $\KK$ of characteristic zero. By the word `variety' we mean a separated algebraic variety over $\KK$.

\subsection{Toric varieties and Cox rings} \label{tvcr}

Let $T = (\KK^{\times})^n$ be an algebraic torus of rank $n$. A~normal irreducible variety $X$ is called \emph{toric}, if there is a faithful action of $T$ on $X$ with an open orbit. There is a well-known combinatorial description of toric varieties. We briefly recall it in this subsection and refer to standard textbooks \cite{CLS, F} on toric varieties for details.

Let $N \cong \ZZ^n$ be the lattice of one-parameter subgroups of $T$ and $M = \Hom_\ZZ(N, \ZZ) \cong \ZZ^n$ be the character lattice of $T$. Denote by $N_{\QQ}$ and $M_{\QQ}$ the associated $\QQ$-vector spaces $N \otimes_\ZZ \QQ$ and $M \otimes_\ZZ \QQ$. Let $\pairing{\cdot}{\cdot} : M_\QQ \times N_\QQ \to \QQ$ be the pairing of dual vector spaces $N_\QQ$ and $M_\QQ$. 

A \emph{cone} in $N$ is a convex polyhedral cone in $N_{\QQ}$. A cone is called \emph{strongly convex} if it contains no nonzero linear subspace. Recall that the \textit{dual cone} $\sigma^\vee$ to a cone $\sigma$ in $N$ is defined by
$$
\sigma^\vee = \{ u \in M_\QQ \ | \ \pairing{u}{v} \geq 0 \ \forall v \in \sigma \}.
$$
Given a strongly convex cone $\sigma$ in $N$ one can consider the finitely generated $\KK$-algebra $\KK[\sigma^\vee \cap M]$ graded by the semigroup of lattice points of the cone $\sigma^\vee$:
$$
\KK[\sigma^\vee \cap M] = \bigoplus_{u \in \sigma^\vee \cap M} \KK \chi^u,
$$
where $\chi^u \cdot \chi^{u'} = \chi^{u + u'}$. Denote by $X(\sigma)$ the corresponding affine variety $\Spec(\KK[\sigma^\vee \cap M])$. It is well-known that given an affine variety $X$ there is a bijection between faithful $T$-actions on $X$ and effective $M$-gradings on $\KK[X]$. The aforementioned $\sigma^\vee \cap M$-grading on $\KK[\sigma^\vee \cap M]$ corresponds to a faithful $T$-action on $X(\sigma)$ with an open orbit. Therefore, $X(\sigma)$ is a toric variety. Moreover, every affine toric variety $X$ arises in this way.

A \textit{fan} in $N$ is a finite collection $\Sigma$ of strongly convex cones in $N$ such that for all $\sigma_1, \sigma_2 \in \Sigma$ every face of $\sigma_1$ is an element of $\Sigma$ and
the intersection $\sigma_1 \cap \sigma_2$ is a face of both $\sigma_1$ and $\sigma_2$. There is a one-to-one correspondence between toric varieties with an acting torus $T$ and fans in~$N$. Namely, given a fan $\Sigma$ in $N$ the corresponding toric variety $X(\Sigma)$ is a union of affine charts $X(\sigma), \sigma \in \Sigma$, where any two charts $X(\sigma_1)$ and $X(\sigma_2)$ are glued along their open subset $X(\sigma_1 \cap \sigma_2)$. Every toric variety arises in this way.

Denote by $|\Sigma|$ the \emph{support} of a fan $\Sigma$ in $N$:
$$
|\Sigma| = \bigcup_{\sigma \in \Sigma} \sigma.
$$
A fan $\Sigma$ in $N$ is called \textit{complete} if $|\Sigma| = N_\QQ$. The corresponding toric variety $X(\Sigma)$ is complete if and only if the fan $\Sigma$ is complete.

\medskip

Let us recall the basic ingredients of the Cox construction; see \cite[Chapter 1]{ADHL} for more details. Let $X$ be a normal variety with free finitely generated divisor class group $\Cl(X)$ and only constant invertible regular functions. Let $\WDiv(X)$ be the group of Weil divisors on $X$ and $K \subseteq \WDiv(X)$ be a subgroup which maps onto $\Cl(X)$ isomorphically. The \emph{Cox ring} of the variety $X$ is defined as a $K$-graded algebra
$$
R(X) = \bigoplus_{D \in K} H^0(X, D),
$$
where $H^0(X, D) = \{ f \in \KK(X)^\times \ | \ \div(f) + D \geq 0 \} \cup \{ 0 \}$, and where the multiplication on homogeneous components coincides with multiplication in $\KK(X)$ and extends to $R(X)$ by linearity. Up to isomorphism the graded ring $R(X)$ does not depend on the choice of a subgroup $K$. This construction may also be generalized to the case when $\Cl(X)$ is a finitely generated group with torsion.

Consider a fan $\Sigma$ in $N$ and the corresponding toric variety $X = X(\Sigma)$. One-dimensional cones in $\Sigma$ are called \textit{rays}. We denote by $\Sigma(1) = \{ \rho_1, \dots, \rho_m \}$ the set of rays in $\Sigma$. Similarly, we denote by $\sigma(1)$ the set of rays of a cone $\sigma$. Let us denote by $p_l$ the generator of the semigroup $\rho_l \cap N$. Assume that $X$ has only constant invertible regular functions. This is equivalent to the condition $\langle p_1, \dots, p_m \rangle_{\QQ} = N_{\QQ}$.

There is a one-to-one correspondence between cones $\sigma \in \Sigma$ and $T$-orbits $O(\sigma)$ in $X$ such that $\dim O(\sigma) = n - \dim \langle \sigma \rangle$. In particular, each ray $\rho_l \in \Sigma(1)$ corresponds to a prime $T$-invariant Weil divisor $D_l = \closure{O(\rho_l)}$ on $X$. The divisor class group $\Cl(X)$ of $X$ is generated by classes $[D_1], \dots, [D_m]$. Moreover, there is a short exact sequence

\begin{equation}\label{eq:M-Cl}
\begin{tikzcd}
0 \arrow[r] & M \arrow[r] & \ZZ^m \arrow[r] & \Cl(X) \arrow[r] & 0,  
\end{tikzcd}
\end{equation}
where the map $M \to \ZZ^m$ is given by $u \mapsto (\pairing{u}{p_1}, \dots, \pairing{u}{p_m})$ and the map $\ZZ^m \to \Cl(X)$ is given by $(c_1, \dots, c_m) \mapsto c_1 [D_1] + \dots + c_m [D_m]$. In~\cite{Cox} it is proved that the Cox ring $R(X)$ is a $\Cl(X)$-graded polynomial algebra $\KK[x_1, \dots, x_m]$ with $\deg(x_l) = [D_l]$ for each $1\le l\le m$.

One can easily characterize radiant toric varieties in terms of Cox rings. Namely, a complete toric variety $X$ is radiant if and only if the group $\Cl(X)$ is free and there is a positive integer $n < m$ such that we can reorder the variables in $\KK[x_1, \dots, x_m]$ in such a way that $\deg(x_{n + 1}),\ldots,\deg(x_m)$ form a basis in $\Cl(X)$ and for $1\le j\le n$ we have $\deg(x_j)=\sum_{k = n + 1}^m a_{k j} \deg(x_k)$ for some non-negative integers $a_{k j}$; cf.~\cite[Corollary~3]{D2}.

\smallskip

The $\Cl(X)$-grading on $R(X)$ gives rise to an action of the quasitorus $G = \Hom_{\ZZ}(\Cl(X), \KK^{\times})$ on $\KK^m = \Spec(R(X))$. Note that if we apply the $\Hom_\ZZ(-, \KK^\times)$-functor to the short exact sequence~\eqref{eq:M-Cl}, we obtain an inclusion $G \hookrightarrow (\KK^{\times})^m$. 

Denote by $Z$ the closed subset of $\KK^m$ defined by equations
$$
\prod_{\rho_l \not\in \sigma(1)} x_l = 0 \quad \forall \sigma \in \Sigma.
$$
These equations generate the so-called {\it irrelevant ideal} in $R(X)$; see \cite[Section~1.6.3]{ADHL} for more details. The subset $Z$ is invariant under the action of the group $G$, thus there is an action of $G$ on the variety $\hat{X} = \KK^m \setminus Z$. The variety $X$ is isomorphic to the good categorical quotient $\hat{X} /\!/ G$ and the isomorphism $X \cong \hat{X} /\!/ G$ is compatible with torus: $X \supset T \cong (\KK^{\times})^m /\!/ G = (\KK^{\times})^m / G$.

%%%%%%%%%%%%%%%%%%%%%%%%%%%%%%%%%%%%%%%%

\subsection{Demazure roots and root subgroups} \label{rss}

Hereinafter we fix a complete fan $\Sigma$ in $N$ with rays $\rho_1, \dots, \rho_m$ and their primitive generators $p_1, \dots, p_m$ and let $X = X(\Sigma)$, $R = R(X)=\KK[x_1,\ldots,x_m]$. Denote by $\R_i \subseteq M, \ 1\le i\le m$ the set of vectors $e \in M$ such that
$$
\pairing{e}{p_i} = -1 \quad \text{and} \quad \pairing{e}{p_s} \geq 0 \ \text{for all} \ s \neq i.
$$
The elements of the set $\R = \bigsqcup_{i=1}^m \R_i$ are called \textit{Demazure roots} of $\Sigma$; cf. \cite[D\'efinition 4]{Dem}. The roots in $\S = \R \cap -\R$ are called \emph{semisimple}, and the ones in $\U = \R \setminus \S$ are called \emph{unipotent}.

\smallskip

Let $e \in \R_i$. Denote by $\partial_e$ the locally nilpotent derivation on $R$:
$$
\partial_e = \left(\prod_{s \neq i} x_s^{\pairing{e}{p_s}} \right) \frac{\partial}{\partial x_i} = x^{\theta (e)} \frac{\partial}{\partial x_i},
$$
where the map $\theta : \R_i \to \ZZ^m$ is defined by
$$
\theta (e) = \big( \pairing{e}{p_1}, \dots, \pairing{e}{p_{i-1}}, 0, \pairing{e}{p_{i+1}}, \dots, \pairing{e}{p_m} \big).
$$
The action of operators $\exp (\alpha \partial_e), \alpha \in \KK$ on $R$ gives rise to an action of the group $U_e \cong \GG_a$ on $\hat{X}$ that commutes with the action of $G$. This induces a $U_e$-action on $\hat{X} /\!/ G = X$ so $U_e \subseteq \Aut (X)$. The subgroups $U_e, \ e \in \R$ are exactly the root subgroups of $\Aut (X)$. Indeed, let $e \in \R$. It is easy to see that as a subgroup of $\Aut (\hat{X})$ the group $U_e$ is normalized by $(\KK^{\times})^m$, so $U_e \subseteq \Aut (X)$ is normalized by the acting torus $T$. Conversely, it is known that every $\GG_a$-subgroup of $\Aut (X)$ which is normalized by $T$ has the form $U_e$ for some $e \in \R$; see \cite[Th\'eor\`eme 3]{Dem} and \cite[Proposition 3.14]{Oda}.

In the case of a radiant toric variety $X$ there is a normalized additive action on $X$ given by the tuple of pairwise commuting homogeneous locally nilpotent derivations $\partial_1, \ldots, \partial_n$ on the ring $R(X) = \KK[x_1, \dots, x_m]$, where
$$
\partial_j = x_{n + 1}^{a_{n + 1 j}} \ldots x_{m}^{a_{m j}} \frac{\partial}{\partial x_j};
$$
see~\cite[Lemma~3.5]{AR}. 
%%%%%%%%%%%%%%%%%%%%%%%%%%%%%%%%%%%%%%%%

\subsection{The automorphism group of a complete toric variety}

In addition to the quotient realization, the Cox ring $R$ also contains information about the automorphism group of $X$. Let $\Aut_g(R)$ be the group of graded automorphisms of the algebra $R$. If $\phi \in \Aut_g(R)$, then $\phi$ commutes with the action of the group $G$ and thus $\phi$ induces an automorphism of $X$. In fact, there is an isomorphism $\Aut^0(X) \cong \Aut_g(R) / G$. Moreover, $\Aut(X)$ is a linear algebraic group generated by the acting torus $T$, the root subgroups $U_e, e \in \R$ and a finite group; see \cite[Section 4]{Cox}.

We recall basic facts on the structure of the group $\Aut(X)$ proved by Demazure~\cite{Dem}; see also \cite[Section~3.4]{Oda}. The torus $T$ is a maximal torus in $\Aut(X)$ and $\R$ is a root system of $\Aut(X)$ with respect to $T$. The unipotent radical $G_u$ of $\Aut^0(X)$ is generated by the subgroups $U_e, e \in \U$, while there exists a connected reductive subgroup $G_{\text{red}}$ with $T$ as a maximal torus such that $\S$ is the root system for $G_{\text{red}}$ and $\Aut^0(X) = G_u \rtimes G_{\text{red}}$. It follows that if we pick a system of positive roots $\S^+$ with $\S = \S^+ \sqcup \S^-$, then the subgroup generated by $G_u$ and all subgroups $U_e, e \in \S^+$ is a maximal unipotent subgroup in $\Aut(X)$.

\smallskip

By~\cite{Br}, for any connected (not necessarily linear) algebraic group $F$ of dimension $n$ there exists a smooth projective variety $X$ of dimension $2 n$ such that the connected component of the group $\Aut(X)$ is isomorphic to $F$. At the same time, the class of connected linear algebraic groups that can be realized as the connected component of the group $\Aut(X)$ of a complete toric variety $X$ is quite constrained. In particular, it is proved in~\cite[Proposition~3.3]{Dem} that the semisimple part of such a group is a group of type $A$. One of the aims of this work is to study further specific properties of the group $\Aut(X)$ in the case of a radiant toric variety~$X$.

%%%%%%%%%%%%%%%%%%%%%%%%%%%%%%%%%%%%%%%%

\subsection{Regular subgroups}
\label{sec1.4}

Recall that an algebraic subgroup $H \subseteq \Aut(X)$ is called regular if it is normalized by the acting torus $T$. In particular, root subgroups in $\Aut(X)$ are precisely regular $\GG_a$-subgroups.

Note that a regular unipotent subgroup $U$ is defined by the root subgroups which are contained in $U$. Denote by $\R(U)$ the set of all roots $e \in \R$ such that $U_e \subseteq U$. We call elements of the set $\R(U)$ the \emph{roots} of the group $U$. Conversely, given a subset $\M \subseteq \R$ such that the subgroups $U_e, e \in \M$ generate a unipotent subgroup, we denote this subgroup by~$U(\M)$. Clearly, if $U$ is a regular unipotent subgroup, then $U(\R(U)) = U$. If $\M = \R(U(\M))$, then we call such a set of roots \emph{saturated}, see Definition~\ref{def:Sat}.

On the other hand, one might have $\R(U(\M)) \supsetneq \M$ for a subset $\M \subseteq \R$. This can happen due to the fact that if $e, e' \in \R$ and $e + e' \in \R$, then $\langle U_e, U_{e'} \rangle \supset U_{e + e'}$.

\smallskip

Let $1 \le i \le m$, $e \in \R_i$, and $\alpha \in \KK$. For $\alpha \in \KK$ denote
$$
u_e(\alpha) = \exp(\alpha \partial_e).
$$
Each $u_e(\alpha)$ is an automorphism of the $\KK$-algebra $R$, and the action of $U_e$ on $\KK^m = \Spec(R)$ is given by
$$
u_e(\alpha) \cdot P = (u_e(\alpha) \cdot x_1(P), \dots, u_e(\alpha) \cdot x_m(P)),
$$
where $P = (c_1, \dots, c_m)$ is a point in $\KK^m$ and $u_e(\alpha) \cdot x_i \in R$ is considered as a function on $\KK^m$ for each $1 \le i \le m$. Since
$$
u_e(\alpha) x_i = x_i + \alpha x^{\theta(e)} \text{ and } u_e(\alpha) x_j = x_j \text{ for } 1 \le j \ne i \le m,
$$
we see that $u_e(\alpha)$ changes only the $i$th coordinate.

Consider another pair $e' \in \R_j$ and $\alpha' \in \KK$. Denote $d = \pairing{e}{p_j}$. Then
$$
u_{e'}(\alpha') u_e(\alpha) x_k = x_k \text{ for all } k \ne i, j.
$$
If $i = j$, then
$$
u_{e'}(\alpha') u_e(\alpha) x_i =
u_{e'}(\alpha')(x_i + \alpha x^{\theta(e)}) =
x_i + \alpha' x^{\theta(e')} + \alpha x^{\theta(e)}.
$$
If $i \ne j$, then
$$
u_{e'}(\alpha') u_e(\alpha) x_j = u_{e'}(\alpha') x_j = x_j + \alpha' x^{\theta(e')} \text { and}
$$

\medskip

$$
u_{e'}(\alpha') u_e(\alpha) x_i =
u_{e'}(\alpha')(x_i + \alpha x^{\theta(e)}) =
u_{e'}(\alpha')\left(x_i + \alpha x_j^d \frac{x^{\theta(e)}}{x_j^d}\right) =
$$
$$
= x_i + \alpha (u_{e'}(\alpha') x_j)^d \frac{x^{\theta(e)}}{x_j^d} =
x_i + \sum_{l = 0}^d \binom{d}{l} x_j^{d - l} \alpha (\alpha')^l x^{l \theta(e')} \frac{x^{\theta(e)}}{x_j^d} =
$$
$$
= x_i + \sum_{l = 0}^d \binom{d}{l} \alpha (\alpha')^l \frac{x^{\theta(e) + l \theta(e')}}{x_j^l}.
$$

\begin{Example}
Let $\Sigma$ be a fan of projective plane $\PP^2$ so that $n=2$, $m = 3$ and ${p_1 = (1, 0)}$, ${p_2 = (0, 1)}$, $p_3 = (-1, -1)$. Take the dual basis in $M$ and consider two roots $\M = \{ (-1, 1), (0, -1) \}$. As $(3 \times 3)$-matrices acting on homogeneous coordinates $[z_0 : z_1 : z_2]$, the elements of the subgroups $U_{(-1, 1)}$ and $U_{(0, -1)}$ are
$$
u_{(-1, 1)}(\alpha) =
\begin{pmatrix}
1 & \alpha & 0 \\
0 & 1 & 0 \\
0 & 0 & 1
\end{pmatrix}
\text{ and } \
u_{(0, -1)}(\beta) =
\begin{pmatrix}
1 & 0 & 0 \\
0 & 1 & \beta \\
0 & 0 & 1
\end{pmatrix},
$$
respectively. Clearly, these root subgroups generate a unipotent subgroup $U(\M)$. However, one can check that
$$
u_{(-1, 1)}(\alpha) u_{(0, -1)}(1) u_{(-1, 1)}(-\alpha) u_{(0, -1)}(-1) =
\begin{pmatrix}
1 & 0 & \alpha \\
0 & 1 & 0 \\
0 & 0 & 1
\end{pmatrix},
$$
which is an element of the subgroup $U_{(-1, 0)}$. It shows that
$$
\R(U(\M)) = \M \cup \{ (-1, 0) \} \supsetneq \M.
$$
\end{Example}

Let $U$ be a regular unipotent subgroup in $\Aut(X)$ with $\R(U) = \{e_1, \ldots, e_k\}$, and $U_{e_1}, \ldots, U_{e_k}$ be all root subgroups in $U$. Then the multiplication map
$$
U_{e_1} \times \ldots \times U_{e_k} \to U, \quad (u_1, \ldots, u_k) \mapsto u_1 \ldots u_k
$$
is an isomorphism of affine varieties. This fact is proved in~\cite[Proposition~28.1]{Hum} for regular unipotent subgroups of a reductive group, but the proof passes without changes in the case we are interested in.

%%%%%%%%%%%%%%%%%%%%%%%%%%%%%%%%%%%%%%%%

\section{Demazure roots of bilateral fans}
\label{sec2}

In this section we continue the study of Demazure roots (or simply roots) of bilateral fans started in~\cite{AR, D1, D2}. Let $N$ be a lattice of rank $n$ and $\Sigma$ be a complete fan in $N$ which is bilateral. Recall that this means that we can order the rays $\rho_1, \ldots, \rho_m$ of $\Sigma$ so that the primitive vectors $p_1, \dots, p_n$ form a basis of $N$ and for any $n+1\le k\le m$ we have
\begin{equation}
\label{bab}
p_k = -\sum_{j = 1}^n a_{k j} p_j, \text{ where } a_{k j} \in \ZZ_{\geq 0}.
\end{equation}

Let $q_1, \dots, q_n$ be the basis of the lattice $M = \Hom(N, \ZZ)$ dual to $p_1, \dots, p_n$. One can easily see that every root $e \in \R_i$, $1\le i\le n$ has the form
$$
e = -q_i + \sum_{j \neq i} b_j q_j, \text{ where } b_j = \pairing{e}{p_j} \in \ZZ_{\geq 0}.
$$
Moreover, we have $-q_i \in \R_i$ for every $1\le i\le n$.

A root $e \in \R$ is called \emph{basic} if $e = -q_i$ for some $1\le i\le n$. We also call a root of the form $-q_i + q_j$ \emph{elementary}, where $1\le i, j\le n$ and $i \ne j$ . All other roots in $\bigsqcup_{i = 1}^n \R_i$ are called \emph{special}. Finally, roots in $\bigsqcup_{k = n + 1}^m \R_k$ are called \emph{detached}.

\smallskip

Note that if $e \in \R_i \cap \S$ for some $1\le i\le n$, then $e$ is either basic or elementary. In \cite[Proposition~2]{D1} it is proved that $\R_k \subseteq \{q_1, \dots, q_n\}$ for all $n+1\le k\le m$. So, the set $\R$ is composed of

\begin{itemize}
\item basic roots in $\bigsqcup_{i = 1}^n \R_i$;
\item elementary roots in  $\bigsqcup_{i = 1}^n \R_i$;
\item special roots in $\bigsqcup_{i = 1}^n \R_i$, they are unipotent;
\item detached roots in $\bigsqcup_{k = n + 1}^m \R_k$, they are semisimple. 
\end{itemize}

\begin{Lemma} \label{lm:RSu} %Roots Sum
Consider roots $e \in \R_i$ and $e' \in \R_j$ for some $1\le i, j\le n$, $i \ne j$ such that $\pairing{e'}{p_i} = 0$. Denote $d = \pairing{e}{p_j}$.
\begin{enumerate}
\item
If $d > 0$, then $e + e' \in \R_i$ and $[\partial_e, \partial_{e'}] = -d \partial_{e + e'}$.
\item
If $d = 0$, then $[\partial_e, \partial_{e'}] = 0$.
\end{enumerate}
\end{Lemma}

\begin{proof}
Let us prove that $e + e' \in \R_i$ when $d > 0$. Clearly, the conditions $\pairing{e + e'}{p_i} = -1$ and $\pairing{e + e'}{p_l} \geq 0$ are satisfied for $l \ne i, j$. We have $\pairing{e + e'}{p_j} = \pairing{e}{p_j} + \pairing{e'}{p_j} = d - 1 \ge 0$ so $e + e' \in \R_i$.

Write $\partial_e = f x_j^d \frac{\partial}{\partial x_i}$ and $\partial_{e'} = h \frac{\partial}{\partial x_j}$ for some monomials $f$ and $h$ that do not depend on $x_i$ and $x_j$. Then $[\partial_e, \partial_{e'}] = f h [x_j^d \frac{\partial}{\partial x_i}, \frac{\partial}{\partial x_j}]$. If $d > 0$, then $[\partial_e, \partial_{e'}] = -d f h x_j^{d-1} \frac{\partial}{\partial x_i} = -d \partial_{e + e'}$. Otherwise $d = 0$ and $[\partial_e, \partial_{e'}] = 0$.
\end{proof}

Now let us consider the following partial preorder on the set $C = \{1, \dots, n\}$. We say that $i \succeq j$ if $a_{k i} \ge a_{k j}$ for all $n+1\le k\le m$, where the integers $a_{kj}$ are defined in~(\ref{bab}). Further, we let $i \asymp j$ if $a_{k i} = a_{k j}$ for all $n+1\le k \le m$. The relation $\asymp$ is an equivalence relation on $C$ and the preorder $\succeq$ induces a partial order on the set of equivalence classes.

One can regard the defined preorder as a preorder on the set of columns of a matrix. Namely, we write the coordinates of the vectors $-p_{n + 1}, \dots, -p_m$ as rows of an $(m - n) \times n$ matrix:
$$
A =
\begin{pmatrix}
a_{n + 1 \, 1} & \hdots & a_{n + 1 \, n} \\
\vdots &  & \vdots \\
a_{m 1} & \hdots & a_{m n}
\end{pmatrix}.
$$
We say that $A$ is the \emph{ray matrix} of the bilateral fan $\Sigma$. Let $v_1, \dots, v_n$ be the columns of the matrix $A$. Then $i \succeq j$ if and only if $v_i \geq v_j$ (coordinate-wise), and $i \asymp j$ if and only if $v_i = v_j$.

The rows of $A$ are nonzero pairwise distinct primitive integer vectors with non-negative entries. Since the fan $\Sigma$ is complete, the matrix $A$ contains no zero column: indeed, if there is $i \in C$ such that $a_{k i} = 0$ for all $k > n$, then every cone from $\Sigma$ is contained in the half-space $\{v \in N_{\QQ} \ | \ \pairing{q_i}{v} \ge 0\}$, a contradiction. Conversely, any matrix $A$ satisfying these conditions is the ray matrix of some bilateral fan $\Sigma$.

The ray matrix $A$ does not define the bilateral fan $\Sigma$. But it determines the set of rays $\Sigma(1)$ and thus the set of Demazure roots $\R$. In fact, it is convenient to describe the set of Demazure roots in terms of the matrix $A$. Let us collect the corresponding results in the following proposition.

\begin{Proposition} \label{pr:RMR} %Ray Matrix Roots
Let $\Sigma$ be a bilateral fan and $A$ be its ray matrix. Then
\begin{enumerate}
\item
a vector $e = -q_i + \sum_{j \ne i} b_j q_j$, $b_j \in \ZZ_{\ge 0}$ belongs to $\R_i$ if and only if $-A e \ge 0$;
\item
a vector $e \in M$ is a detached root in $\R_k$ if and only if $e = q_i$, where the column $v_i$ is a vector of length $(m - n)$ with a $1$ in the $(k - n)^\text{th}$ coordinate and $0$'s elsewhere.
\end{enumerate}
\end{Proposition}

\begin{proof}
\leavevmode
\begin{enumerate}
\item
The equality $\langle e, p_i\rangle = -1$ follows from the definition of a dual basis. Since 
$$
-A e = v_i - \sum_{j \ne i} b_j v_j, 
$$
the condition $-A e \ge 0$ is just the system of conditions $\pairing{e}{p_k} \ge 0$, $k > n$ written in vector form.
\item
Let $e = \sum_{j = 1}^n w_j q_j$. Then $e \in \R_k$ for some $k > n$ if and only if all $w_j = \pairing{e}{p_j}$ are non-negative and the vector $A e = \sum_{j = 1}^n w_j v_j$ has one coordinate 1 and all other coordinates non-positive. Since all vectors $v_j$ are nonzero and have non-negative coordinates, this is the case if and only if all coordinates $w_j$ but one are $0$'s, one coordinate $w_i$ equals~$1$, and the vector $v_i$ also has one coordinate equal to $1$ and others equal to $0$. Then $e = q_i$ and $\pairing{e}{p_k} = -1$ means that $v_i$ has $1$ in the $(k-n)^\text{th}$ coordinate.
\end{enumerate}
\end{proof}

\begin{Example} \label{ex:CEx} %Common Example
Let $X = \FF_1 \times \PP^1$, where $\FF_1$ is the Hirzebruch surface. Taking a basis $p_1,p_2,p_3$ in $N$, we let $p_4=-p_1-p_2$, $p_5=-p_1$ and $p_6=-p_3$, and the (bilateral) fan $\Sigma$ of $X$ is the product of the fan of $\FF_1$ generated by $p_1,p_2,p_4,p_5$ and the fan of $\PP^1$ generated by $p_3,p_6$. We have $n=3$ and $m=6$. The ray matrix of $\Sigma$ is 
$$
A =
\begin{pmatrix}
1 & 1 & 0 \\
1 & 0 & 0 \\
0 & 0 & 1
\end{pmatrix}.
$$
Using Proposition~\ref{pr:RMR}, one can find that
$$
\R_1 = \{-q_1, \ -q_1 + q_2\}, \ \R_2 = \{-q_2\}, \ \R_3 = \{-q_3\},
$$
$$
\R_4 = \{q_2\}, \ \R_5 = \emptyset, \ \R_6 = \{q_3\}.
$$

Consider a fan $\Sigma'$ obtained from $\Sigma$ by deforming the ray $\QQ_{\ge 0} \cdot p_5$ into $\QQ_{\ge 0} \cdot p'_5$, where $p'_5 = -p_1 - p_3$. That is, for each cone $\Cone(S)$ generated by a subset $S \subseteq \{p_1, \dots, p_6\}$, we consider a cone $\Cone(S')$, where $S'$ is obtained from $S$ by substituting $p_5$ with $p'_5$. The collection $\Sigma'$ of such cones $\Cone(S')$ is also a bilateral fan. Its ray matrix is equal to
$$
A' =
\begin{pmatrix}
1 & 1 & 0 \\
1 & 0 & 1 \\
0 & 0 & 1
\end{pmatrix}.
$$
Consider a fan $\Sigma''$ obtained from $\Sigma'$ by gluing its cones $\Cone(p_3, p_4, p'_5)$ and $\Cone(p_2, p_3, p'_5)$. In other words, we remove the two aforementioned cones and their intersection $\Cone(p_3, p'_5)$ from the set $\Sigma$, add $\Cone(p_2, p_3, p_4, p'_5)$ to this new set, and denote the resulting set by $\Sigma''$. The set $\Sigma''$ is a bilateral fan and its ray matrix is also equal to $A'$.
\end{Example}

\medskip

The next observation follows from Lemma~\ref{lm:RSu}; it is also proved in~\cite[Lemma~2]{D2}.

\begin{Lemma} \label{lm:er} % elementary roots 
For any $1\le i\le n$ if $\R_i\ne\{-q_i\}$, then $\R_i$ contains an elementary root. Moreover, if $\pairing{e}{p_j} >0$ for some $e\in\R_i$ and $1\le j\le n$, then $-q_i+q_j\in\R_i$. 
\end{Lemma}

We proceed with the following result. 

\begin{Proposition} \label{pr:RPO} %Roots PreOrder
Take $i, j \in C$ with $i \ne j$. Then
\begin{enumerate}
\item
$i \succeq j$ if and only if $\R_i$ contains the elementary root $-q_i + q_j$;
\item
if $i \asymp j$, then the translation map $\xi : M \to M$ given by $\xi(e) = -q_i + q_j + e$ induces a bijection $\R_j \setminus \{-q_j + q_i\} \to \R_i \setminus \{-q_i + q_j\}$.
\end{enumerate}
\end{Proposition}

\begin{proof}
\leavevmode
\begin{enumerate}
\item
By Proposition~\ref{pr:RMR}.(1) the vector $-q_i + q_j$ is a root if and only if $-A (-q_i + q_j) = v_i - v_j \ge 0$ and by definition $i \succeq j$ if and only if $v_i \ge v_j$.
\item
The conditions $\langle e,p_i\rangle=-1$ and $\langle e,p_s\rangle\ge 0$ for any $s\ne i$ mean that every root in $\R_i$ has the form $e=-q_i+\sum_{s\ne i} b_sq_s$ with 
$b_s\in\ZZ_{\ge 0}$. Similarly, every root in $\R_j$ has the form $e'=-q_j+\sum_{s\ne j} b_s'q_s$ with $b_s'\in\ZZ_{\ge 0}$. Since we have $p_k=-\sum_{s=1}^n a_{ks}p_s$ for $n+1\le k\le m$, the remaining conditions on elements in $\R_i$ are  
$$
\langle e,p_k\rangle=a_{ki}-a_{kj}b_j-\sum_{s\ne i,j}a_{ks}b_s\ge 0.
$$
Similarly, we have
$$
\langle e',p_k\rangle=a_{kj}-a_{ki}b_i'-\sum_{s\ne i,j}a_{ks}b_s'\ge 0.
$$
By assumptions, we know that $a_{ki}=a_{kj}$ for all $n+1\le k\le m$. So if $b_j=b_i'=0$, the conditions imposed on the sets $\{b_s\}$ and $\{b_s'\}$ are the same, and it remains to prove that if $e \in \R_i$ and $\pairing{e}{p_j}> 0$, then $e = -q_i + q_j$. We have
$$
-A e = v_i - \sum_{s \ne i} b_s v_s = -(b_j - 1) v_j - \sum_{s \ne i, j} b_s v_s
$$
since $v_i = v_j$ by definition of $\asymp$. We obtain that a non-negative vector $-A e$ equals the sum of non-positive vectors. It follows that $(b_ j - 1) v_j = 0$ and $b_s v_s = 0$ for all 
$s \ne i, j$. But all $v_s$ are nonzero, so $b_j = 1$ and $b_s = 0$ for all $s \ne i, j$. We obtain $e = -q_i + q_j$.
\end{enumerate}
\end{proof}

The following lemma contains one more property of Demazure roots that will be needed in the next section.

\begin{lemma} \label{lm:SSR} %SemiSimple Roots
For $i \in C$ we have $q_i \in \R$ if and only if all roots in $\R_i$ are semisimple. In particular, the set $\R_i$ contains no special root.
\end{lemma}

\begin{proof}
If the root $-q_i \in \R_i$ is semisimple, then $q_i \in \R$.

Conversely, assume that $q_i \in \R$. Then $q_i$ is detached and the vector $v_i$ is as in Proposition~\ref{pr:RMR}.(2). Let $e = -q_i + \sum_{s \ne i} b_s q_s$ be a root in $\R_i$. Then
$$
-A e = v_i - \sum_{s \ne i} b_s v_s \ge 0.
$$
Since all vectors $v_s$ are nonzero, $\sum_{s \ne i} b_s$ equals either $0$ or $1$. In the first case all $b_s$ equal~$0$ and the root $e=-q_i$ is semisimple. Otherwise we have $v_i = v_j$ and $e = -q_i + q_j$ for some $1 \le j \ne i \le n$, so $i \asymp j$. In this case the vector $-e$ is a root as well, so the root $e$ is semisimple.
\end{proof}

Let $C_1, \dots, C_r$ be the equivalence classes of the relation $\asymp$. By Proposition~\ref{pr:RPO} an elementary root $-q_i + q_j$ is semisimple if and only if $i$ and $j$ are in the same equivalence class. By renumbering the classes we may assume that $C_l \succ C_s$ implies $l < s$. Furthermore, by renumbering the vectors $p_1, \dots, p_n$ we may assume that $C_1, \dots, C_r$ is an (element-wise) increasing sequence of segments of $C$. In other words, we have the following:
\begin{equation}\label{eq:Cs-condition}
\begin{array}{c}
\text{the sequence } C_1, \dots, C_r \text{ is } \succeq \text{-non-increasing \ and} \\
\text{there exist numbers } 0 = c_0 < c_1 < \ldots < c_r = n \text{ such that for all } 1\le s\le r \\
\text{we have } C_s = \{c_{s - 1} + 1, c_{s - 1} + 2, \dots, c_s\}.
\end{array}
\end{equation}

\begin{Lemma} \label{lm:RSt} %Roots Structure
Assume that the condition~\eqref{eq:Cs-condition} holds. If $e \in \R_i$ is a unipotent root, then $e = -q_i + \sum_{j > i} \pairing{e}{p_j} q_j$.
\end{Lemma}

\begin{proof}
The case of basic $e$ is trivial. Assume that $\pairing{e}{p_j} > 0$ for some $j \in C$. By Lemma~\ref{lm:er}, we have $-q_i + q_j \in \R_i$. So $i \succeq j$ by Proposition~\ref{pr:RPO}.(1). 

If $i \asymp j$, then by Proposition~\ref{pr:RPO}.(2) and the fact that $e$ is unipotent we have $e + q_i - q_j \in \R_j$. It implies $\pairing{e}{p_j} = 0$, a contradiction. We conclude that $i \succ j$ and thus $j>i$. 
\end{proof}

%%%%%%%%%%%%%%%%%%%%%%%%%%%%%%%%%%%%%%%%

\section{The structure of a maximal unipotent subgroup}
\label{sec3}

Let us fix a radiant toric variety $X = X(\Sigma)$ with a bilateral fan $\Sigma$ and keep the notation of the previous section. Let us also assume the condition~\eqref{eq:Cs-condition}. We are going to describe the structure of a maximal unipotent subgroup of the automorphism group $\Aut(X)$.

Given $v \in N$ such that $\pairing{e}{v} \neq 0 \ \forall e \in \S$, we denote
$$
\S^+_v = \{e \in \S \ | \ \pairing{e}{v} > 0\} \text{ and } \R^+_v = \S^+_v \sqcup \U.
$$
Recall that $U(\R^+_v)$ is a maximal (regular) unipotent subgroup in $\Aut (X)$ and $\R(U(\R^+_v)) = \R^+_v$. Since any system of positive roots in $\S$ is $\S^+_v$ for some $v\in N$, 
any maximal regular unipotent subgroup is obtained this way for a suitable $v$. Let us define $\Umax = U(\R^+_v)$ with $v$ such that
\begin{equation}\label{eq:v-qi}
\pairing{q_1}{v} < \dots < \pairing{q_n}{v} < 0.
\end{equation}

\begin{Lemma} \label{lm:PRS} %Positive Roots Structure
Assume that the condition~\eqref{eq:Cs-condition} is satisfied and $v \in N$ is chosen so that~\eqref{eq:v-qi} holds. Then a root $e \in \R$ lies in $\R^+_v$ if and only if it has the form $e = -q_i + \sum_{j > i} \pairing{e}{p_j} q_j$ for some~$i \in C$.
\end{Lemma}

\begin{proof}
All unipotent roots lie in $\R^+_v$ and have the desired form by Lemma~\ref{lm:RSt}. A semisimple root is basic, elementary or detached. In the first case it is in $\R^+_v$ and has the desired form. In the second case $-q_i+q_j\in\R^+_v$ if and only if $j>i$ due to~\eqref{eq:v-qi}. In the third case it is not in  $\R^+_v$ and is not of the corresponding form. 
\end{proof}

Hereinafter we fix a vector $v$ satisfying the condition~\eqref{eq:v-qi}. Denote $\R^+ = \R^+_v$ and $\R_i^+ = \R_i \cap \R^+$ for all $i \in C$. In particular, we have $\R^+ = \R(\Umax) = \bigsqcup_{i=1}^n \R_i^+$.

\begin{Remark} \label{rm:MRS} %Minimal Roots Structure
By Lemma~\ref{lm:er}, if for $i\in C$ there is a root $e\in\R_i$ with $\pairing{e}{p_j} > 0$ for some $j \in C$, then $i \succeq j$. We conclude that if $i$ lies in a $\succeq$-minimal equivalence class $C_s$ and $i$ is a maximal number in $C_s$, then $\R_i^+ = \{-q_i\}$. In particular, $\R_n^+ = \{-q_n\}$.
\end{Remark}

Given a subset $C' \subseteq C$, we denote
$$
U(C') = U\left( \bigsqcup_{i \in C'} \R_i^+ \right).
$$
For example, $U(i) = U(\R_i^+) = \langle U_e \ | \ e \in \R_i^+ \rangle$. Note that if $C' = \{i_1, \dots, i_k\}$, then $U(C')$ acts trivially on coordinates $x_j$ in the Cox ring $R(X)$ for all $j \ne i_1, \dots, i_k$; see Subsection~\ref{rss}. In particular, if $I$ and $J$ are disjoint subsets of $C$, then $U(I) \cap U(J) = \{\id\}$. Recall that for $e \in \R$ and $\alpha \in \KK$ we  denote $u_e(\alpha) = \exp(\alpha \partial_e) \in U_e$.

\begin{Lemma} \label{lm:MoR} %Multiplication of Roots
Let $1 \le i \le j \le n$, $e \in \R_i^+$, $e' \in \R_j^+$, and $\alpha, \alpha' \in \KK$. Denote $d = \pairing{e}{p_j}$.
\begin{enumerate}
\item
The sum $e + e'$ is a root if and only if $i < j$ and $d > 0$, in which case $e + e' \in \R_i$.

%%%%%%%%%%%%%%%

\item
If $i < j$, then
$$
u_{e'}(\alpha') u_e(\alpha) \cdot x_i =
\prod_{l = 0}^d u_{e + l e'}\left(\binom{d}{l} \alpha (\alpha')^l\right) \cdot x_i.
$$

%%%%%%%%%%%%%%%

\item
If $e + e'$ is not a root, then $u_e(\alpha)$ commutes with $u_{e'}(\alpha')$.
\end{enumerate}
\end{Lemma}

%%%%%%%%%%%%%%%%%%%%%%%%%%%%%%

\begin{proof}
\leavevmode
\begin{enumerate}
\item
Assume that $e + e'$ is a root. Then $i < j$, since otherwise $i = j$ and $\pairing{e + e'}{p_i} = -2$. Thus, $e + e' \in \R_i$ by Lemma~\ref{lm:PRS} and $d - 1 = \pairing{e + e'}{p_j} \ge 0$. The converse follows from Lemma~\ref{lm:RSu}.

%%%%%%%%%%%%%%%

\item
The action of the product $u_{e'}(\alpha') u_e(\alpha)$ on $x_i$ was computed in Subsection~\ref{sec1.4}:
$$
u_{e'}(\alpha') u_e(\alpha) x_i =
x_i + \sum_{l = 0}^d \binom{d}{l} \alpha (\alpha')^l \frac{x^{\theta(e) + l \theta(e')}}{x_j^l}.
$$
By the first statement of the lemma, $e + l e' \in \R_i$ for each $0 \le l \le d$. Note that
$$
\frac{x^{\theta(e) + l \theta(e')}}{x_j^l} = x^{\theta(e + l \theta(e'))}.
$$
Therefore,
$$
u_{e'}(\alpha') u_e(\alpha) x_i =
x_i + \sum_{l = 0}^d \binom{d}{l} \alpha (\alpha')^l x^{\theta(e + l \theta(e'))} =
\prod_{l = 0}^d u_{e + l e'}\left(\binom{d}{l} \alpha (\alpha')^l\right) \cdot x_i.
$$

%%%%%%%%%%%%%%%

\item
An automorphism $u \in \Umax$, arising from an automorphism of the Cox ring $R(X)$, is defined by its action on the coordinate functions $x_s$, $1 \le s \le m$ of $R(X)$. Therefore, if $u_1, u_2 \in \Umax$, then $u_1 = u_2$ if and only if
$$
u_1 x_s = u_2 x_s \text{ for all } 1 \le s \le m.
$$

Denote $u_1 = u_{e'}(\alpha') u_e(\alpha)$ and $u_2 = u_e(\alpha) u_{e'}(\alpha')$. By the computations in Subsection~\ref{sec1.4}, we know
$$
u_1 x_k = x_k = u_2 x_k \text{ for } k \ne i, j,
$$
so we only have to consider the cases $s = i, j$.

If $i = j$, then $u_e(\alpha)$ clearly commutes with $u_{e'}(\alpha')$. Let $i \ne j$. Since $e + e'$ is not a root, $d = 0$ by the first statement of the lemma. So by the second statement of the lemma
$$
u_1 x_i = x_i + \alpha x^{\theta(e)} = u_e(\alpha) x_i = u_2 x_i.
$$
The monomial $\theta(e')$ does not contain the variable $x_i$ since $i < j$, so
$$
u_2 x_j = u_{e'}(\alpha') x_j = u_1 x_j.
$$
Therefore, $u_1 = u_2$.
\end{enumerate}
\end{proof}

\begin{Proposition} \label{pr:Nor} %Normalization
\leavevmode
\begin{enumerate}
\item
The subgroup $U(i)$ is commutative of dimension $|\R_i^+|$.
\item
The subgroup $U(i)$ is normalized by the subgroup $U(j)$ for $j > i$. Moreover, if $e \in \R_i^+$, $e' \in \R_j^+$, and $d = \pairing{e}{p_j}$, then
$$
u_{e'}(\alpha') u_e(\alpha) u_{e'}(\alpha')^{-1} = \prod_{k=0}^d u_{e + k e'}\left( \binom{d}{k} \alpha (\alpha')^k \right) \in U(i)
$$
for all $\alpha, \alpha' \in \KK$.
\end{enumerate}
\end{Proposition}

\begin{proof}
The first assertion follows from Lemma~\ref{lm:MoR}. Denote
$$
u_1 = u_{e'}(\alpha') u_e(\alpha) u_{e'}(\alpha')^{-1} =
u_{e'}(\alpha') u_e(\alpha) u_{e'}(-\alpha'),
$$
$$
u_2 = \prod_{l = 0}^d u_{e + l e'}\left(\binom{d}{l} \alpha (\alpha')^l\right).
$$
Clearly, $u_2 \in U(i)$, so $u_2$ fixes $x_s$ for all $s \ne i$.

As in Lemma~\ref{lm:MoR}, it suffices to prove $u_1 x_s = u_2 x_s$ for all $1 \le s \le m$. This equation clearly holds for all $s \ne i, j$. Since the monomial $\theta(e')$ does not contain the term $x_i$, we have
$$
u_1 x_j = u_{e'}(\alpha') u_{e'}(-\alpha') x_j = x_j = u_2 x_j.
$$
By Lemma~\ref{lm:MoR} $u_1 x_i = u_2 x_i$ since
$$
u_1 x_i = u_{e'}(\alpha') u_e(\alpha) x_i.
$$
Therefore, $u_1 = u_2$.
\end{proof}

\begin{Remark}
Since all elements of the subgroup $U(i)$ change only one coordinate $x_i$ in the spectrum $\KK^m = \Spec(R(X))$ of the Cox ring $R(X) = \KK[x_1, \ldots, x_m]$ (see Subsection~\ref{sec1.4}), we conclude that for any $1 \le i \le n$ general orbits of the action of the subgroup $U(i)$ on $X$ are one-dimensional.
\end{Remark}

Let $U \subseteq \Aut(X)$ be a regular unipotent subgroup such that $\M = \R(U) \subseteq \R^+$. For each $i = 1, \dots, n$ denote $\M_i = \M \cap \R_i^+$. 

\begin{Corollary} \label{cr:SDG} %SemiDirect of \GG_a
We have
$$
U = \Big( \big( \cdots ( U(\M_n) \ltimes U(\M_{n - 1}) ) \dots \big) \ltimes U(\M_2) \Big) \ltimes U(\M_1).
$$
In particular, $\Umax \cong \Big( \big( \cdots ( \GG_a \ltimes \GG_a^{|\R_{n-1}^+|} ) \dots \big) \ltimes \GG_a^{|\R_2^+|} \Big) \ltimes \GG_a^{|\R_1^+|}$.
\end{Corollary}

\begin{proof}
Applying the second claim of Proposition~\ref{pr:Nor} to all pairs $(1, j)$, $j > 1$ we obtain $U = U(\M_2 \cup \ldots \cup \M_n) \ltimes U(\M_1)$. Similarly, $U(\M_2 \cup \ldots \cup \M_n) = U(\M_3, \ldots,\M_n) \ltimes U(\M_2)$, and so on. At the end we obtain
$$
U = \Big( \big( \cdots ( U(\M_n) \ltimes U(\M_{n - 1}) ) \dots \big) \ltimes U(\M_2) \Big) \ltimes U(\M_1),
$$
which proves the assertion. If $U = \Umax$, then $\M = \R^+$ and $U(\M_i) = U(i) = \GG_a^{|\R_i^+|}$ by the first assertion of Proposition~\ref{pr:Nor}. Moreover, $|\R_n^+| = 1$ by Remark~\ref{rm:MRS}.
\end{proof}

\medskip

Now let us fix positive integers $k > l$ and consider the subgroup $U_{k, l}$ of the group $U_k$ of all unitriangular $(k \times k)$-matrices such that
$$
(\beta_{i j})_{i, j = 1}^k \in U_{k, l} \iff \beta_{i j} = 0 \text{ for all } j > i > l.
$$
In other words, the group $U_{k, l}$ consists of matrices of the form

\smallskip

$$
\renewcommand{\bigstar}{\text{\Huge *}}
\newcommand{\bigzero}{\text{\huge 0}}
\left( \begin{array}{c|c}
\begin{matrix}
1 &  &  & \\
 & \ddots & \bigstar & \\
 &  & \ddots &  \\
\bigzero &  &  & 1
\end{matrix} &
\bigstar \\
\hline
\bigzero &
\begin{matrix}
1 &  &  & \\
 & \ddots & & \bigzero \\
 &  & \ddots &  \\
\bigzero &  &  & 1
\end{matrix}
\end{array} \right)
$$

\bigskip

In particular, the subgroup $U_{k, k - 1}$ coincides with the group $U_k$. Note also that the group $U_{k, l}$ is commutative if and only if $l = 1$.

\begin{Proposition} \label{pr:ECG} %Equivalence Class Group
Let $C_s = \{c_{s - 1} + 1, c_{s - 1} + 2, \dots, c_s\}$ be an equivalence class in $C$. Denote $c = c_{s - 1} + 1$. Then
$$
U(C_s) \cong U_{k, l}
$$
with $k = |\R_c^+| + 1$ and $l = |C_s|$.
\end{Proposition}

\begin{proof}
Denote by $E$ the identity $(k \times k)$-matrix. Let $E_{i, j}$ be a $(k \times k)$-matrix with a unique nonzero entry on $i^{\text{th}}$ row and $j^{\text{th}}$ column equal to $1$. For each $i \le l$ the matrices $E + \sum_{j > i} \beta_{i j} E_{i, j}$ form a subgroup $H_i$ in $U_{k, l}$ isomorphic to $\GG_a^{k - i}$:
$$
\left( E + \sum_{j > i} \beta_{i j} E_{i, j} \right) \cdot \left( E + \sum_{j > i} \beta'_{i j} E_{i, j} \right) = E + \sum_{j > i} (\beta_{i j} + \beta'_{i j}) E_{i, j}.
$$

Further, for $l \ge i' > i$ we have
$$
\left( E + \sum_{j > i'} \beta_{i' j} E_{i', j} \right) \cdot \left( E + \sum_{j > i} \beta_{i j} E_{i, j} \right) = E + \sum_{j > i'} \beta_{i' j} E_{i', j} + \sum_{j > i} \beta_{i j} E_{i, j},
$$
so every element of $U_{k, l}$ is uniquely decomposed as the product $h_l h_{l - 1} \cdots h_1$, where $h_i \in H_i$ for all $i$. Finally, for $1 \le i < j \leq k$ and $i < i' < j' \leq k$ write
$$
(E - \beta' E_{i', j'}) (E -  \beta E_{i, j}) (E - \beta' E_{i', j'})^{-1} = (E - \beta' E_{i', j'}) (E -  \beta E_{i, j}) (E + \beta' E_{i', j'}) =
$$
$$
= E - \beta E_{i, j} - \delta_{j i'} \beta \beta' E_{i, j'}.
$$
In particular, $U_{k, l} = \Big( \big( \cdots (H_l \ltimes H_{l - 1}) \dots \big) \ltimes H_2 \Big) \ltimes H_1$.

By Corollary~\ref{cr:SDG} the group $U(C_s)$ is a semidirect product of the subgroups $U(c)$, $U(c + 1)$, $U(c + 2), \dots$ with the action on each next normalized subgroup as in Proposition~\ref{pr:Nor}. Note that if $i \in C_s$, then $|\R_i^+| = |\R_c^+| - (i - c) = k - (i - c + 1)$ by Proposition~\ref{pr:RPO}.(2) and Lemma~\ref{lm:PRS}.

Denote by $\T$ the set of roots in $\R^+$, which are not elementary semisimple. Let us enumerate the roots $e \in \R_c^+ \cap \T$ by numbers $j_e$, $l + 1 \le j_e \le k$. By Proposition~\ref{pr:RPO}.(2), for every $i \in C_s$ the map $e \mapsto e - q_c + q_i$ establishes a bijection between roots in $\R_i^+ \cap \T$ and roots in $\R_c^+ \cap \T$. We let $j_e$ be equal to $j_{e'}$, where the root $e'$ corresponds to $e$ under this bijection.

We conclude that a map $\phi: U(C_s) \to U_{k, l}$ defined on generators as
$$
\phi(u_{-q_i + q_j}(\alpha)) = E - \alpha E_{i - c + 1, j - c + 1} \text{ with } i, j \in C_s;
$$
$$
\phi(u_e(\alpha)) = E - \alpha E_{i - c + 1, j_e} \text{ for other roots } e \in \R_i^+
$$
is the required isomorphism.
\end{proof}

\begin{Theorem} \label{th:MUS} %Maximal Unipotent Subgroup
Let $\Umax$ be a maximal unipotent subgroup of the automorphism group $\Aut(X)$ of a radiant toric variety $X$. Assume that the condition~\eqref{eq:Cs-condition} holds for the equivalence classes $C_1, \dots, C_r$ of the relation $\asymp$. Then
$$
\Umax \cong \Big( \big( \cdots \big( U_{k_r, l_r} \ltimes U_{k_{r-1}, l_{r-1}} \big) \dots \big) \ltimes U_{k_2, l_2} \Big) \ltimes U_{k_1, l_1},
$$
where $k_s = |\R_{c_{s - 1} + 1}^+| + 1$ and $l_s = |C_s|$, $1\le s\le r$.
\end{Theorem}

\begin{proof}
By the same argument as in the proof of Corollary~\ref{cr:SDG}, grouping together subsets $\R_i^+$, $i \in C_s$ for each $1\le s\le r$, we obtain a decomposition
$$
\Umax = \Big( \big( \cdots \big( U(C_r) \ltimes U(C_{r - 1}) \big) \dots \big) \ltimes U(C_2) \Big) \ltimes U(C_1).
$$
By Proposition~\ref{pr:ECG} we have $U(C_s) \cong U_{k_s, l_s}$. This proves the theorem.
\end{proof}

\begin{Example}
Let $X$ be the projective space $\PP^n$ so that the number of rays of the fan $\Sigma$ equals $n + 1$ and $p_{n + 1} = -(p_1 + \dots + p_n)$. We have $1 \asymp 2 \asymp \dots \asymp n$, so the whole set $C$ is an equivalence class. For each $i \in C$ the set of roots $\R_i^+$ consists of vectors 
$$
-q_i, -q_i + q_{i + 1}, -q_i + q_{i + 2} \dots, -q_i + q_n. 
$$
In particular, we have $|\R_1^+| = n$, so $\Umax \cong U_{n + 1, n}$ is the group $U_{n + 1}$. Of course, this also follows from the well-known fact $\Aut(\PP^n) = \PGL(n + 1)$.
\end{Example}

\begin{Example}
Let $X$ be as in Example~\ref{ex:CEx}. Then 
$$
\Umax = (U(3) \ltimes U(2)) \ltimes U(1) \cong (\GG_a \ltimes \GG_a) \ltimes \GG_a^2. 
$$
In fact, for this particular case $U(1, 2) \cong U_3$ and the subgroup $U(3)$ commutes with the subgroup $U(1, 2)$, so we have $\Umax = U(3) \times U(1, 2) \cong \GG_a \times U_3$.
\end{Example}

Finishing this section, let us describe a maximal unipotent subgroup $\Uss$ of the semisimple part of the group $\Aut(X)$ as a subgroup of $\Umax$. The roots of $\Uss$ are the semisimple roots in $\R^+$. For any subset $S \subseteq C$ denote by $\R(S)$ the set of roots~$\bigsqcup_{i \in S} \R_i$ and let $\R^+(S) = \R(S) \cap \R^+$. For each $1\le s\le r$ let $U^s \subseteq U(C_s)$ be the subgroup such that $\R(U^s) = \R^+(C_s) \cap \S$. By Theorem~\ref{th:MUS} there is a decomposition
$$
\Uss = \Big( \big( \cdots \big( U^r \ltimes U^{r - 1} \big) \dots \big) \ltimes U^2 \Big) \ltimes U^1.
$$

Denote by $U_{k, l}^{(1)}$ (resp. $U_{k, l}^{(2)}$) the subgroup of $U_{k, l}$ consisting of all matrices that differ from the unit matrix at most in the upper left $(l \times l)$-square (resp. at most in the upper right $(l \times (k - l))$-rectangle). Then $U_{k, l}=U_{k, l}^{(1)} \ltimes U_{k, l}^{(2)}$, the subgroup $U_{k, l}^{(1)}$ is isomorphic to $U_l$, and the subgroup $U_{k, l}^{(2)}$ is commutative.

\begin{Proposition}
\leavevmode
In the above notation, the following holds.
\begin{enumerate}
\item
$\Uss = U^1 \times \ldots \times U^r$.
\item
For each $1\le s\le r$ we have $U^s = U(C_s) \cong U_{k_s}$ if all roots in $\R(C_s)$ are semisimple, and $U^s \cong U_{k_s, l_s}^{(1)}\cong U_{l_s}$ otherwise.
\end{enumerate}
\end{Proposition}

\begin{proof}
\leavevmode
\begin{enumerate}
\item
Since we already have a semidirect product decomposition of $\Uss$, it suffices to prove that each two subgroups $U^s$ and $U^{s'}$, $s \ne s'$ commute with each other. It follows from Lemma~\ref{lm:MoR} and the facts that a positive semisimple root is either basic or elementary and the sum of two such roots from different components is not a root.

\item
If all roots in $\R(C_s)$ are semisimple, then $U^s = U(C_s) \cong U_{k_s, l_s}\cong U_{k_s}$, since $l_s = k_s - 1$ by definition. Assume that there are unipotent roots in $\R(C_s)$. By Lemma~\ref{lm:SSR} for each $i \in C_s$ we have $-q_i \in \U$ and due to the construction of the isomorphism $U(C_s) \cong U_{k_s, l_s}$ in the proof of Proposition~\ref{pr:ECG} we see that $U^s \cong U_{k_s, l_s}^{(1)}$.
\end{enumerate}
\end{proof}

We see that each subgroup $U^s$ is the group of all unitriangular matrices: it is isomorphic either to $U_{k_s}$ or to $U_{l_s}$. This is a manifestation of a famous theorem due to Demazure: the semisimple part of the automorphism group of a complete toric variety is a group of type~$A$; see~\cite[Proposition~3.3]{Dem}. The number of simple components in the semisimple part of the automorphism group $\Aut(X)$ equals the number of classes $C_s$ such that either $C_s$ contains at least two elements or the corresponding basic root $-q_i$ is semisimple. 

The center of $\Umax$ is described in Remark~\ref{rm:CU} below.

%%%%%%%%%%%%%%%%%%%%%%%%%%%%%%%%%%%%%%%%

\section{Regular unipotent subgroups}
\label{sec4}

In this section we study regular unipotent subgroups of the automorphism group $\Aut(X)$ of a radiant toric variety $X$ that act on $X$ with an open orbit. We keep the notation of the previous sections.

\begin{definition} \label{def:Sat} %Saturated
We say that a subset of roots $\mathcal{A} \subseteq \R$ is \emph{saturated} with respect to a subset $\mathcal{B} \subseteq \R$ if for any $a \in \mathcal{A}$ and $b \in \mathcal{B}$ the sum $a + b$ is either contained in $\mathcal{A}$ or is not a root.
\end{definition}

Recall that $\R(\Umax)=\R^+$.

\begin{Lemma} \label{propus} % Saturated steps 
A subset $\M \subseteq \R^+$ equals $\R(U)$ for some regular subgroup $U \subseteq \Umax$ if and only if for each $i \in C$ the subset $\M_i$ is saturated with respect to $\bigsqcup_{j > i} \M_j$.
\end{Lemma}

\begin{proof}
Observe that a subset $\M \subseteq \R^+$ equals $\R(U)$ for some regular subgroup $U \subseteq \Umax$ if and only if the linear span of derivations $\partial_e$, $e\in\M$ is a Lie subalgebra in the tangent algebra of the group $\Umax$. By Lemmas~\ref{lm:RSu} and~\ref{lm:PRS}, this is the case if and only if for all $e\in \M_i$ and $e'\in \M_j$, $j> i$ with $\langle e,p_j\rangle>0$ we have $e+e'\in\M_i$.
\end{proof}

\begin{definition}
A \emph{principal unipotent subgroup} of the automorphism group $\Aut(X)$ of a radiant toric variety $X$ is a regular commutative unipotent subgroup that acts on $X$ with an open orbit.
\end{definition}

The action of a subgroup of $\Aut(X)$ on $X$ is faithful, so a principal unipotent subgroup in $\Aut(X)$ is precisely a subgroup that defines a normalized additive action on $X$. Since every two such actions are equivalent (see \cite[Theorem 3.6]{AR}), every two principal unipotent subgroups in $\Aut(X)$ are conjugate. As we have seen in Subsection~\ref{rss}, $U(\{-q_1, \ldots, -q_n\})$ is an example of a principal unipotent subgroup.

\begin{Theorem} \label{Th:PU} %principal unipotent subgroup
Let $X$ be a radiant toric variety. A regular unipotent subgroup $U$ in $\Umax$ with $\M=\R(U)$ acts on $X$ with an open orbit if and only if $U$ contains a principal unipotent subgroup. More precisely, the subgroup $U$ acts on $X$ with an open orbit if and only if $-q_1,\ldots,-q_n\in\M$. 
\end{Theorem} 

\begin{proof}
If a subgroup $U(\M)$ acts on $X$ with an open orbit, then the set $\M_i$ is nonempty for each~$i$. Indeed, if such intersection is empty, then any orbit of the group $U(\M)$ on the spectrum of the Cox ring $R(X)$ is contained in the subvariety, where the coordinates $x_i,x_{n+1},\ldots,x_m$ are constant. So, the dimension of any orbit is less than $n$, and it can not project to an open orbit on $X$. 

Let us prove by induction that $-q_1, \ldots, -q_n \in \M$. By Remark~\ref{rm:MRS}, we have $-q_n \in \M$. Assuming for some $i<n$ that $-q_{i+1},\ldots,-q_n\in \M$, we take an arbitrary element of $\M_i$, say $-q_i + \sum_{j > i} b_j q_j$, and apply consecutively Lemma~\ref{lm:RSu} for each $e'=-q_{i+1}, \ldots, -q_n$. Then $-q_i \in \M$, and the direct implication follows. 

In Subsection~\ref{rss} we have already seen that locally nilpotent derivations $\partial_{-q_1}, \dots, \partial_{-q_n}$ give rise to an additive action on $X$, so the inverse implication is straightforward.
\end{proof} 

This fact and the results of~\cite{D2} motivate the following definition.

\begin{definition}
We say that a radiant toric variety $X$ is of \emph{Type I}, if a maximal unipotent subgroup in $\Aut(X)$ is commutative. All other radiant toric varieties are assigned to \emph{Type~II}.
\end{definition}

In Type~I, the maximal unipotent subgroup in $\Aut(X)$ is a principal unipotent subgroup. By \cite[Corollary~4]{D2}, radiant toric varieties of Type~I are precisely complete toric varieties that admit a unique additive action. On the other hand, Type~II consists exactly of complete toric varieties that admit a faithful action of a non-commutative unipotent group with an open orbit. 

\smallskip

Clearly, the direct product $X_1\times\ldots\times X_s$ of radiant toric varieties is again a radiant toric variety, and it is of Type~I if and only if all factors $X_i$ are of Type~I. Let us give one more result in this direction.

\begin{Proposition}
Every radiant toric variety of Type~I admits a unique decomposition ${X \cong (\PP^1)^b \times Y}$, where $b$ is a non-negative integer and $Y$ is a radiant toric variety of Type~I such that
the connected component $\Aut(Y)^0$ is solvable.
\end{Proposition}

\begin{proof}
Since $X$ is of Type~I, besides basic roots $-q_1, \ldots, -q_n$ the set $\R$ may contain only some detached roots $q_i$.

Assume that $q_1 \in \R_{n + 1}$. By Proposition~\ref{pr:RMR}.(2), in the ray matrix $A$ we have $a_{n + 1 \, 1} = 1$ and $a_{k \, 1} = 0$ for all $n+2\le k\le m$. If $a_{n + 1\, s} > 0$ for some $2\le s\le n$, then by Proposition~\ref{pr:RMR}.(1) we obtain $-q_s + q_1 \in \R_s$, a contradiction. So, we have $a_{n + 1 \, s} = 0$ for all $2\le s\le n$.

It means that $\Sigma = \Sigma_1 \oplus \Sigma_2$, where $\Sigma_1$ is the fan generated by $p_1$ and $-p_1$ and $\Sigma_2$ is a complementary subfan. It implies that $X \cong \PP^1 \times X(\Sigma_2)$. Repeating this procedure, we come to the decomposition $X \cong (\PP^1)^b \times Y$, where the fan of the toric variety $Y$ has no detached roots. We conclude that $\Aut(Y)^0 \cong T \ltimes U$, where $T$ is the acting torus of $Y$ and $U$ is a principal unipotent subgroup of $\Aut(Y)$. In particular, the group $\Aut(Y)^0$ is solvable.

In any decomposition of this form, $b$ is the number of pairs of semisimple roots of $\Sigma$, and the fan of $Y$ is the subfan of $\Sigma$ obtained as an intersection with the subspace generated by the vectors $p_i$ such that $\R_i$ contains no semisimple root. This shows that the decomposition is unique.
\end{proof}

Let us give an example of a smooth non-projective radiant toric variety $X(\Sigma)$ of dimension~$n$ for each $n \geq 3$; see \cite{Sha} for details.

\begin{Example}
\label{exsha}
The set of primitive vectors on the rays of $\Sigma$ consists of a basis $p_1, \dots, p_n$ of the lattice $N$, the vector $w = -p_1 - \ldots - p_n$, and the vectors $u_i = p_i + w$, $i \in C = \{1, \dots, n\}$. The maximal cones of $\Sigma$ are constructed as follows. For each integer $k$ denote by $p_k$ and $u_k$ the vectors $p_i$ and $u_i$ correspondingly, where $k$ and $i$ are equal modulo $n$ and $i \in C$. Now for each $i \in C$ write a sequence of vectors
$$
p_i, p_{i+1}, \dots, p_{i + n - 2}, u_i, u_{i+1}, \dots, u_{i + n - 2}, w.
$$
For each (continuous) segment $S$ of $n$ elements of this sequence, the cone generated by elements of $S$ lies in $\Sigma$. Lastly, $\Sigma$ also contains the cone generated by $p_1, \ldots, p_n$.

The toric variety $X(\Sigma)$ is of Type I. Indeed, the ray matrix $A$ equals 
$$
\begin{pmatrix}
0 & 1 & \dots & 1 \\
1 & 0 & \dots & 1 \\
\vdots &  & \ddots & \vdots \\
1 & 1 & \dots & 0 \\
1 & 1 & \dots & 1 \\
\end{pmatrix},
$$
so each two columns of $A$ are incomparable. In the case $n = 3$ the orbit structure of the action of $\Umax$ on $X$ is described in \cite{Sha}. This description shows that the number of orbits of a maximal unipotent subgroup $\Umax$ on a radiant toric variety can be infinite.
\end{Example}

\begin{Problem}
Characterize bilateral fans $\Sigma$ such that the action $\Umax \curvearrowright X(\Sigma)$ has a finite number of orbits.
\end{Problem}

Since a closed subvariety of a projective variety is projective, the direct product of complete varieties is projective if and only if each factor is projective. This shows that the product $X \times \PP^2$, where $X$ is a variety from Example~\ref{exsha}, is an example of a smooth non-projective radiant toric variety of Type~II.

\smallskip

Now we come to an algorithm that allows to construct all regular unipotent subgroups in $\Aut(X)$ that act on $X$ with an open orbit.

\begin{Algorithm} \label{alg}
We sequentially construct a subset $\M \subseteq \R^+$, which is the set of roots of a regular unipotent subgroup of $\Aut(X)$. For this we produce subsets $\M_i \subseteq \R_i^+$ step by step starting from $i = n$ down to $i = 1$. By Theorem~\ref{Th:PU} we start with $\M_n = \R_n^+ = \{-q_n\}$.

Assume that we constructed subsets $\M_{i + 1}, \M_{i + 2}, \dots, \M_n$ for some $i\in C$ such that for each $j > i$ the set $\M_j$ is saturated with respect to $\bigsqcup_{k > j} \M_k$. Let us construct the set $\M_i$. Once again by Theorem~\ref{Th:PU} we have to start with $\M_i = \{-q_i\}$. Choose some subset of roots $E \subseteq \R_i^+$ that we wish to add to $\M_i$. Since we want to preserve the condition of saturation, we also need to add to $\M_i$ all sums of the form $e + \sum_s e_s$, where $e \in E$, $e_s \in \bigsqcup_{j > i} \M_j$, and $e + \sum_s e_s \in \R_i^+$. After that we move to $\M_{i - 1}$. In short, we iterate over the collection of subsets of $\R_i^+$ which are saturated with respect to $\bigsqcup_{k > i} \M_k$ and recursively repeat the procedure for $\M_{i - 1}$.

Assume that we constructed $\M_1, \dots, \M_n$. Let $\M = \M_1 \sqcup \ldots \sqcup \M_n$. By Theorem~\ref{Th:PU} there exists a regular unipotent subgroup $U \subseteq \Aut(X)$ such that $\R(U) = \M$ and $U$ acts on $X$ with an open orbit. All regular unipotent subgroups that act on $X$ with an open orbit are obtained this way for a suitable choice of intermediate subsets $\M_i$.
\end{Algorithm}

\begin{Example} \label{ex:P123}
Let $n = 3$, $m = 4$, and $p_4 = (-3, -2, -1)$. In this case $X$ is a weighted projective space $\PP(1,2,3)$. We have $1 \succ 2 \succ 3$ and
$$
\R_3^+ = \{ -q_3 \},
$$
$$
\R_2^+ = \{ -q_2, \ -q_2 + q_3, \ -q_2 + 2 q_3 \},
$$
$$
\R_1^+ = \{ -q_1, \ -q_1 + q_2, \ -q_1 + q_2 + q_3, \ -q_1 + q_3, \ -q_1 + 2 q_3, \ -q_1 + 3 q_3 \}.
$$
By Corollary~\ref{cr:SDG}, we obtain $\Umax \cong \big( \GG_a \ltimes \GG_a^3 \big) \ltimes \GG_a^6$.

Let us apply Algorithm~\ref{alg}. We start with $\M_3 = \{ -q_3 \}$.  The following cases are possible for~$\M_2$.

\begin{enumerate}
\item If $\M_2 = \{ -q_2 \}$, then a subset $\M_1 \subseteq \R_1^+$ satisfies the constraint
$$
-q_1 + b_2 q_2 + b_3 q_3 \in \M_1 \implies
$$
$$
-q_1 + b'_2 q_2 + b'_3 q_3 \in \M_1 \ \text{for all} \ b'_2 \leq b_2, \ b'_3 \leq b_3.
$$
This case gives 11 subgroups: one of dimension 3, two of dimension 4, two of dimension~5, three of dimension 6, two of dimension 7, and one of dimension 8.

\item If $\M_2 = \{ -q_2, -q_2 + q_3 \}$, then $\M_1 \subseteq \R_1^+$ satisfies the constraint from the first case and two additional ones:
$$
-q_1 + q_2 \in \M_1 \implies -q_1 + q_3 \in \M_1,
$$
$$
-q_1 + q_2 + q_3 \in \M_1 \implies -q_1 + 2 q_3 \in \M_1.
$$
This case gives 9 subgroups: two for each of dimensions 6, 7, 8 and one for each of dimensions 4, 5, and 9.

\item If $\M_2 = \R_2^+$, then there are two additional constraints on $\M_1 \subseteq \R_1^+$:
$$
-q_1 + q_2 \in \M_1 \implies -q_1 + 2 q_3 \in \M_1,
$$
$$
-q_1 + q_2 + q_3 \in \M_1 \implies -q_1 + 3 q_3 \in \M_1.
$$
This case gives 7 subgroups: two of dimension 8 and one for each of dimensions 5, 6, 7, 9, and 10.
\end{enumerate}
We conclude that there are 27 regular unipotent subgroups in $\Umax$ that act on $\PP(1, 2, 3)$ with an open orbit.
\end{Example}

%%%%%%%%%%%%%%%%%%%%%%%%%%%%%%%%%%%%%%%%

\section{Centers and central series}
\label{sec5}

We proceed with a computation of the center of a regular unipotent subgroup. Let $U$ be a regular unipotent subgroup of $\Aut(X)$ that acts on $X$ with an open orbit. Denote 
$$
C(U) = \{i \in C \mid \pairing{e}{p_i} \le 0 \ \forall e \in \R(U)\}.
$$

\begin{Proposition}
The center $Z(U)$ of $U$ equals $\prod_{i\in C(U)}U_{-q_i}$. In particular, $Z(U)$ is contained in a principal unipotent subgroup.
\end{Proposition}

\begin{proof}
Consider the center $\mathfrak{z} (\u)$ of the tangent algebra $\u = \Lie (U)$. 
Let $\delta = \sum_{e \in \R(U)} \alpha_e \partial_e \in \mathfrak{z} (\u)$. 
Then by Lemma~\ref{lm:RSu} for all $i \in C$ we have
$$
0 = [\delta, \partial_{-q_i}] = \sum_{e : \pairing{e}{p_i} > 0} -\alpha_e\pairing{e}{p_i} \partial_{e - q_i}.
$$
So, $\alpha_e = 0$ for any non-basic $e$, hence $\delta = \sum_{i = 1}^n \alpha_{-q_i} \partial_{-q_i}$. Furthermore, if $-q_j + q_k \in \R(U)$ for some $j \ne k$, then by Lemma~\ref{lm:RSu}
\[
0 = [\delta, \partial_{-q_j + q_k}] = [\alpha_{-q_k} \partial_{-q_k}, \partial_{-q_j + q_k}] = \alpha_{-q_k} \pairing{-q_j + q_k}{p_k} \partial_{-q_j} = \alpha_{-q_k} \partial_{-q_j}.
\]
By Lemma~\ref{lm:er},  we have $\delta = \sum_{i \in C(U)} \alpha_i \partial_{-q_i}$. Conversely, every $\delta \in \u$ of such form lies in $\mathfrak{z}(\u)$. 
We conclude with the fact that $\Lie(Z(U)) = \mathfrak{z}(\u)$.
\end{proof}

\begin{Remark} \label{rm:CU} %Center of Umax
If $U = \Umax$, then $C(U)$ is the set of minimal numbers of classes $C_s$ that are maximal with respect to the partial preorder $\succeq$.
\end{Remark}

In the rest of this section we describe the lower and upper central series of a regular unipotent subgroup $U$ in terms of a directed graph on the set of roots of $U$.

\begin{definition}
Given a regular unipotent subgroup $U \subseteq \Umax$ with $\M = \R(U)$, we define a directed graph $\Gamma(\M)$ with $\M$ as the set of vertices as follows. There is an arrow from a root $a \in \M$ to a root $b \in \M$ if and only if there exists $e \in \M$ such that $b = a + e$.

Let $k$ be a non-negative integer. We denote by $\M^{\uparrow k}$ the subset of roots $b$ in $\M$ such that there is a path of length $k$ in $\Gamma(\M)$ ending in $b$. Further, let us denote by $\M^{\downarrow k}$ the subset of roots $a$ in $\M$ such that any path in $\Gamma(\M)$ starting in $a$ is of length less than $k$.
\end{definition}

\begin{Remark} \label{rm:GA} %Graph Arrow
If there is an arrow from $a \in \M_i$ to $b = a + e \in \M_j$, then $i \ge j$ and $[\partial_a, \partial_e] = d \partial_b$ for some $d \neq 0$ by Lemma~\ref{lm:RSu}.
\end{Remark}

We claim that the graph $\Gamma(\M)$ is acyclic. Assume to the contrary that there is a cycle $S = (b_0 = a, b_1, \dots, b_{l - 1}, b_l = a)$. Then each vertex of this cycle lies in the same $\M_i$. For each $1\le k\le l$ denote $e_k = b_k - b_{k - 1}$. We have $a = a + \sum_{k = 1}^l e_k$ or, equivalently, $\sum_{k = 1}^l e_k = 0$. This is a contradiction, since due to Lemma~\ref{lm:PRS} the sum of positive roots cannot be zero.

\begin{Proposition} \label{pr:CS} %Central Series
Let $U$ be a regular unipotent subgroup and $l$ be the length of the longest path in $\Gamma(\M)$, where $\M = \R(U)$. Then the lower (resp. upper) central series of $U$ consists of subgroups $U(\M^{\uparrow k})$ (resp. $U(\M^{\downarrow k})$) for $0\le k\le l + 1$ in this order.
\end{Proposition}

\begin{proof}
We proceed by induction on $k$. If $k = 0$, then $\M^{\uparrow k} = \M$ and $\M^{\downarrow k} = \emptyset$, so $U(\M^{\uparrow k})=U$ and $U(\M^{\downarrow k})=\{\id\}$ are the first subgroups of the corresponding central series. Assume that $U(\M^{\uparrow k})$ and $U(\M^{\downarrow k})$ are the $k^{\text{th}}$ subgroups of the lower and upper central series, respectively. We use the fact that the tangent algebra of the commutator of two closed connected subgroups is the commutator of their tangent algebras; see~\cite[Theorem~24.5.11]{TY}.  So the tangent algebra $\Lie([U, U(\M^{\uparrow k})])$ is spanned by all elements $[\partial_e, \partial_a]$, where $e \in \M$ and $a \in \M^{\uparrow k}$. By Remark~\ref{rm:GA}, they span the subspace $\langle \partial_b \mid b \in \M^{\uparrow k + 1} \rangle$, hence $[U, U(\M^{\uparrow k})] = U(\M^{\uparrow k + 1})$.

Similarly, consider $h \in U$ such that $ghg^{-1}h^{-1}\in U(\M^{\downarrow k})$ for all $g \in U$. Then $h = \exp(\partial)$ for some $\partial \in \Lie(U)$ and $[\partial, \partial_e] \in \Lie(U(\M^{\downarrow k}))$ for all $e \in \M$. So, $\partial$ is a linear combination of elements $\{\partial_u \mid u \in \M^{\downarrow k + 1}\}$.

Finally, $\M^{\uparrow l} \supsetneq \M^{\uparrow l + 1} = \emptyset$ and $\M^{\downarrow l}\subsetneq \M^{\downarrow l + 1} = \M$.
\end{proof}

\begin{Corollary}
The nilpotency class of $\Umax$ equals $l + 1$, where $l$ is the length of the longest path in $\Gamma(\R^+)$.
\end{Corollary}

\begin{Remark}
The upper and lower central series of a regular unipotent group $U$ may not coincide. Indeed, let $\M_2=\{-q_2\}$ and $\M_1=\{-q_1,-q_1+q_2, -q_1+2q_2\}$. Then $\M^{\uparrow1}=\{-q_1,-q_1+q_2\}$, but $\M^{\downarrow2}=\{-q_1,-q_1+q_2, -q_2\}$.
\end{Remark}

\begin{Corollary}
The derived length of $\Umax$ equals the minimal number $k$ such that
$$
\underbrace{\Big( \cdots \big( (\R^+)^{\uparrow 1} \big)^{\uparrow 1} \cdots \Big)^{\uparrow 1}}_{k \text{ times}} = \emptyset
$$
\end{Corollary}

\begin{proof}
The claim is a direct implication of the equality $\R([U, U]) = \R(U)^{\uparrow 1}$ for any regular subgroup $U \subseteq \Umax$.
\end{proof}

In particular, $\Umax$ is commutative if and only if the graph $\Gamma(\R^+)$ has no arrow, and $\Umax$ is metabelian if and only if the graph $\Gamma((\R^+)^{\uparrow 1})$ has no arrow.

\begin{Example} \label{Ex:NC} %Nilpotency class
For the maximal unipotent subgroup $\Umax$ treated in Example~~\ref{ex:P123} the longest path has length $4$ (see Figure~1), so the nilpotency class is $5$. The derived series consists of $\Umax$, $U(\R')$ with $\R'$ in Figure~2, and $U(\{-q_1,-q_1+q_3\})$, so the derived length equals $3$.
\end{Example} 

\begin{figure}[ht] 
\begin{tikzpicture}[scale=2.5,shorten >=3pt]
  \tikzstyle{vertex}=[circle,fill=black,minimum size=5pt,inner sep=0pt]
  \tikzstyle{arrow}=[-latex, dashed, thick]
  \tikzstyle{arrow2}=[-latex, dotted, thick]

\clip (-4.3,-2.0) rectangle (1.9,2.05);
\begin{scope}[y=(120:1),rotate=120]
\foreach \x  [count=\j from 2] in {-8,...,8}
  \draw[help lines, dotted]
    (\x,-8) -- (\x,8)
    (-8,\x) -- (8,\x) 
    [rotate=120] (\x,-8) -- (\x,8) ;
\draw[->] (0,0) -- (-2,0) node[right] {$q_1$};
\draw[->] (0,0) -- (0,-2.3) node[right] {$q_2$};
\draw[->] (0,0) -- (4,4) node[left] {$q_3$};
\coordinate (d-00)   at (1,0);
\coordinate (d-10)   at (1,-1);
\coordinate (d-11)   at (2,0);
\coordinate (d-01)   at (2,1);
\coordinate (d-02)   at (3,2);
\coordinate (d-03)   at (4,3);
\coordinate (d0-0)   at (0,1);
\coordinate (d0-1)   at (1,2);
\coordinate (d0-2)   at (2,3);
\coordinate (d00-)   at (-1,-1);
\foreach \x in {d-00, d-10, d-11, d-01, d-02, d-03, d0-0, d0-1, d0-2, d00-}
  \node[vertex] at (\x) {};

\draw[arrow] (d0-1) -- (d0-0);
\draw[arrow] (d0-2) -- (d0-1);
\draw[arrow] (d-03) -- (d-02);
\draw[arrow] (d-02) -- (d-01);
\draw[arrow] (d-01) -- (d-00);
\draw[arrow] (d-11) -- (d-10);
\draw[arrow] (d-11) -- (d-03);
\draw[arrow] (d-11) -- (d-02);
\draw[arrow] (d-11) -- (d-01);
\draw[arrow] (d-10) -- (d-02);
\draw[arrow] (d-10) -- (d-01);
\draw[arrow] (d-10) -- (d-00);

\draw[arrow2] (d00-) -- (d0-0);
\draw[arrow2] (d00-) -- (d0-1);
\draw[arrow2] (d00-) -- (d-02);
\draw[arrow2] (d00-) -- (d-01);
\draw[arrow2] (d00-) -- (d-00);
\draw[arrow2] (d00-) -- (d-10);
\draw[arrow2] (d0-2) -- (d-03);
\draw[arrow2] (d0-1) -- (d-02);
\draw[arrow2] (d0-0) -- (d-01);
\draw[arrow2] (d0-2) -- (d-02);
\draw[arrow2] (d0-1) -- (d-01);
\draw[arrow2] (d0-0) -- (d-00);

\node at (-2,-1.5) {$\R_3$};
\node at (3.5,1.5) {$\R_1$};
\node at (.5,2.5) {$\R_2$};

\node[above right] at (d00-) {$-q_3$};
\node[below] at (d0-0) {$-q_2$};
\node[below left,align=right] at (d-00) {$-q_1$};
\node[below,align=left] at (d0-1) {$-q_2+q_3$};
\node[below,align=left] at (d0-2) {$-q_2+2q_3$};
\node[below left,align=right] at (d-01) {\contour{white}{$-q_1+q_3$}};
\node[below left,align=right] at (d-02) {$-q_1+2q_3$};
\node[below left,align=right] at (d-03) {$-q_1+3q_3$};
\node[above,align=right] at (d-10) {$-q_1+q_2$};
\node[above,align=right] at (d-11) {$-q_1+q_2+q_3$};
\end{scope}
\end{tikzpicture}

\caption{The graph $\Gamma(\R)$ from Example~\ref{ex:P123}. Here each third of the plane represents the elements of $\R_i$, $i = 1, 2, 3$ in suitable coordinates. The inner arrows are depicted dashed and outer ones dotted.}
\end{figure}
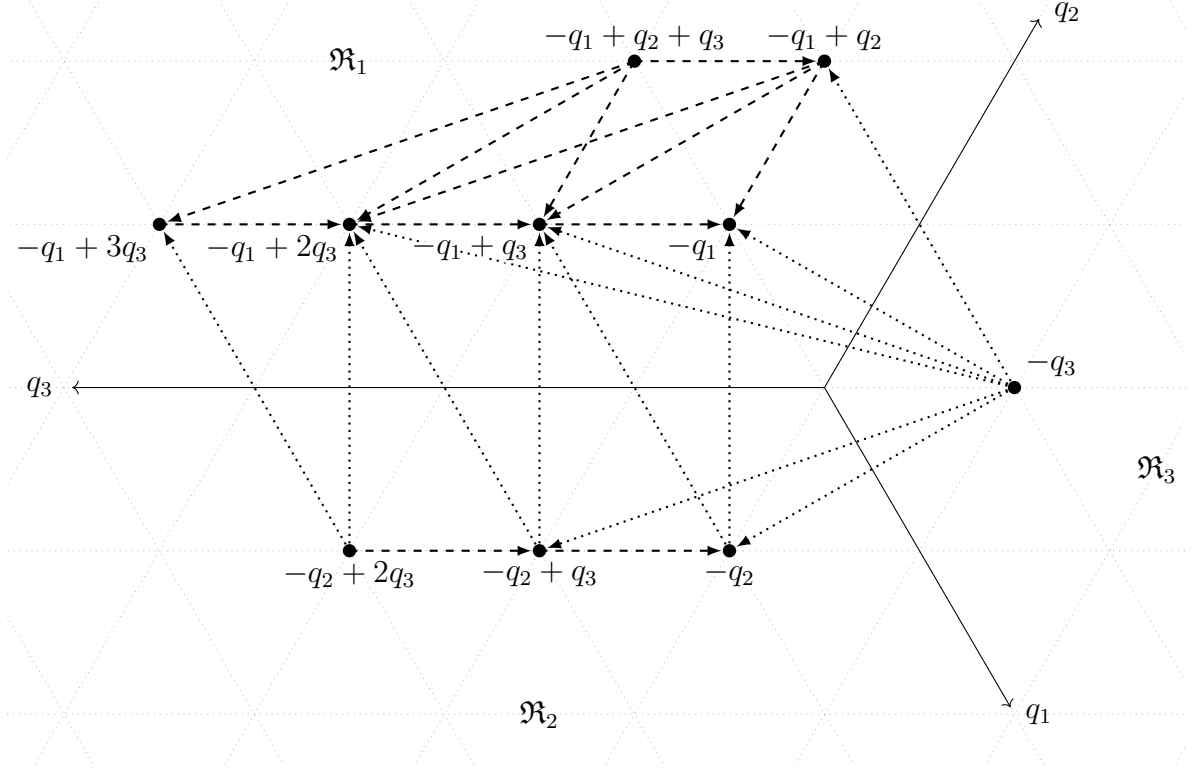

\begin{figure}[ht] 
\begin{tikzpicture}[scale=2.5,shorten >=3pt]
  \tikzstyle{vertex}=[circle,fill=black,minimum size=5pt,inner sep=0pt]
  \tikzstyle{arrow}=[-latex, dashed, thick]
  \tikzstyle{arrow2}=[-latex, dotted, thick]

\clip (-4.3,-2.0) rectangle (1.9,2.05);
\begin{scope}[y=(120:1),rotate=120]
\foreach \x  [count=\j from 2] in {-8,...,8}
  \draw[help lines, dotted]
    (\x,-8) -- (\x,8)
    (-8,\x) -- (8,\x) 
    [rotate=120] (\x,-8) -- (\x,8) ;
\draw[->] (0,0) -- (-2,0) node[right] {$q_1$};
\draw[->] (0,0) -- (0,-2.3) node[right] {$q_2$};
\draw[->] (0,0) -- (4,4) node[left] {$q_3$};
\coordinate (d-00)   at (1,0);
\coordinate (d-10)   at (1,-1);
\coordinate (d-01)   at (2,1);
\coordinate (d-02)   at (3,2);
\coordinate (d-03)   at (4,3);
\coordinate (d0-0)   at (0,1);
\coordinate (d0-1)   at (1,2);
\foreach \x in {d-00, d-10, d-01, d-02, d-03, d0-0, d0-1}
  \node[vertex] at (\x) {};

\draw[arrow] (d-10) -- (d-01);
\draw[arrow] (d-10) -- (d-00);

\draw[arrow2] (d0-1) -- (d-01);
\draw[arrow2] (d0-0) -- (d-00);

\node at (-2,-1.5) {$\R_3$};
\node at (3.5,1.5) {$\R_1$};
\node at (.5,2.5) {$\R_2$};

\node[below] at (d0-0) {$-q_2$};
\node[below left,align=right] at (d-00) {$-q_1$};
\node[below,align=left] at (d0-1) {$-q_2+q_3$};
\node[below,align=right] at (d-01) {\contour{white}{$-q_1+q_3$}};
\node[below,align=right] at (d-02) {$-q_1+2q_3$};
\node[below,align=right] at (d-03) {$-q_1+3q_3$};
\node[above,align=right] at (d-10) {$-q_1+q_2$};
\end{scope}
\end{tikzpicture}

\caption{The graph $\Gamma(\R')$ with $\R'=\R([\Umax,\Umax])$ in Example~\ref{ex:P123}.}
\end{figure}

We end this section with the observation that the lower central series is defined by a subgraph of $\Gamma(\M)$, which has simpler structure.

\begin{definition}
We call an arrow in $\Gamma(\M)$ from $a$ to $b$ \emph{inner}, if $a, b \in \M_i$ for some $i$, and \emph{outer} otherwise.
\end{definition}

\begin{lemma} \label{lm:in-out}
Assume that there is an inner arrow $a \to b$ and an outer one $b \to c$ in $\Gamma(\M)$. Then there exist an outer arrow $a\to b'$ and an inner one $b'\to c$ for some vertex $b' \in \Gamma(\M)$.
\end{lemma}

\begin{proof}
Let $a, b \in \M_j$, $c \in \M_k$, and the root $d = b - a$ belong to $\M_i$. Then $i > j > k$ and the root $e = c - b$ belongs to $\M_k$.

Let us take $b' = a + e$ and prove that $b' \in \M_k$. The root $d$ has $-1$ at $i^{\text{th}}$ coordinate and $0$'s at $j^{\text{th}}$ and $k^{\text{th}}$ coordinates. The roots $a, b$ have $-1$ at $j^{\text{th}}$ coordinate and $0$ at $k^{\text{th}}$ coordinate. The roots $c, e$ have $-1$ at $k^{\text{th}}$ coordinate and $e$ has positive $j^{\text{th}}$ coordinate. These observations imply that $b'$ has $-1$ at $k^{\text{th}}$ coordinate and non-negative elsewhere.

So, by Proposition~\ref{pr:RMR}, since $-A b' = (-A a) + (-A e)$ has non-negative coordinates, $b'$ is indeed a root in $\M_k$. The claim follows.
\end{proof}

\begin{Proposition} \label{PrPr}
For any path ending in a vertex $v \in \Gamma(\M)$ there exists a path of the same length, which ends in $v$ and consists of inner arrows.
\end{Proposition}

\begin{proof}
Consider a path $a_1 \to \cdots \to a_k \to v$ of length $k$, which contains an outer arrow. By Lemma~\ref{lm:in-out}, we can swap inner and outer arrows until the first arrow $a_1 \to a_2^\prime$ is outer.

Let $a_2^\prime = a_1 + a_1^\prime$. Then the arrow $a_1^\prime \to a_2^\prime$ is inner since $a_1^\prime$ and $a_2^\prime$ lie in the same $\R_i$. We obtain a path from $a_1^\prime$ to $v$ of length $k$ with a smaller number of outer arrows. Proceeding in this manner, we construct a desired path.
\end{proof}

\begin{Corollary}
The subgroups $U(\M^{\uparrow i})$ do not change if we substitute the initial graph $\Gamma(\M)$ by the subgraph of inner arrows.
\end{Corollary}

\begin{Remark}
Proposition~\ref{PrPr} shows that there is a (commutative) subgroup $U(\M_i)$ in $U(\M)$ such that the intersections of the subgroups in the lower central series with $U(\M_i)$ are pairwise distinct. In particular, the nilpotency class of $U(\M)$ does not exceed $\max_i|\M_i|$.     
\end{Remark} 

%%%%%%%%%%%%%%%%%%%%%%%%%%%%%%%%%%%%%%%%

\section{Toric equivariant completions of unipotent groups}
\label{sec6}

In this section we consider equivariant completions of unipotent groups by toric varieties. The following proposition shows that a complete toric variety $X$ can not be realized as an equivariant completion of a non-commutative regular unipotent subgroup.

\begin{Proposition}
Consider a complete toric variety $X$ with an acting torus $T$. Let a regular unipotent subgroup $U$ of $\Aut(X)$ act on $X$ with an open orbit $O$ so that the action $U \curvearrowright O$ has trivial stabilizers. Then the group $U$ is commutative.
\end{Proposition}

\begin{proof}
We may assume that $U$ is contained in $\Umax$. If $U$ is not commutative, then by Theorem~\ref{Th:PU} $U$ contains a principal unipotent subgroup as a proper subgroup. By Lemma~\ref{lm:er} there is an elementary root $-q_i + q_j$ in $\R(U)$. Then the one-parameter subgroups $U_{-q_i}$ and $U_{-q_i + q_j}$ in $U$ acting on the spectrum of the Cox ring $R(X)$ change only the coordinate~$x_i$. This shows that generic orbits of $U_{-q_i}$ and $U_{-q_i + q_j}$ on $X$ coincide, and the action $U \curvearrowright O$ has non-trivial stabilizers, a contradiction. 
\end{proof}

The next example demonstrates that the regularity condition is crucial.

\begin{Example}
For $n \ge 3$ we consider the subgroup
$$
U=\left\{\left.
\left( \begin{array}{cccccc}
1 & 0 & 0 & 0 &\cdots  & 0\\
a_1& 1 & 0 & 0 &\cdots  & 0 \\
a_2& a_n& 1 & 0 &\cdots &0 \\
a_3& 0 &  0   & 1&\cdots  &0 \\
\vdots& &      &  & \ddots \\
a_n& 0 &     0 & 0 && 1 \\
\end{array} \right)
\right| a_1,\ldots,a_n\in\KK\right\}\subset\GL(n + 1, \KK).
$$
with the action on $\PP^n$ given in homogeneous coordinates by
\[
(a_1,\ldots,a_n)\cdot [z_0:\ldots:z_n]=[z_0:z_1+a_1z_0:z_2+a_2z_0+a_nz_1:z_3+a_3z_0:\ldots:z_n+a_nz_0].
\]
This subgroup is neither commutative nor regular. It acts on $\PP^n$ with an open orbit and generic stabilizers are trivial.  
\end{Example}

The theory of equivariant completions of commutative unipotent groups is nothing but the theory of additive actions. In the toric case this theory is rather well-developed, see~\cite{AR, AZ, D1, D2, HT}. It will be interesting to study (toric) equivariant completions of non-commutative unipotent groups. Observations given above may serve as first steps towards this goal.

\smallskip

It is natural here to draw an analogy with the case of flag varieties. Take a simple linear algebraic group $G$ and a parabolic subgroup $P$ in $G$. The homogeneous space $G / P$ is a projective variety called a flag variety of the group $G$. This variety contains an open Bruhat cell $O$, which is an orbit of the unipotent radical $U$ of the opposite parabolic subgroup $P^-$, and the action of $U$ on $O$ has trivial stabilizers. So $G / P$ is an equivariant completion of~$U$. Moreover, $U$ may be regarded as a regular unipotent subgroup of $\Aut(G / P)$ since $U$ is normalized by a maximal torus in $G$.

At the same time, if $G$ is the connected component of the group $\Aut(G / P)$ (this is almost always the case), then $G / P$ admits an additive action if and only if the group $U$ is commutative, and the list of such cases is very short (see~\cite{Ar}). In particular, it may happen only if the parabolic subgroup $P$ is maximal.

It is proved in~\cite[Theorem~1.1]{Ch} that $G / P$ can be realized as an equivariant completion of the unipotent radical $U$ in a unique way up to isomorphism, provided $G / P$ is not isomorphic to a projective space.

%%%%%%%%%%%%%%%%%%%%%%%%%%%%%%%%%%%%%%%%

\section{The case of toric surfaces}
\label{sec7} 

In this section we illustrate the results obtained above in the case of surfaces. We keep the notation of the previous sections. Let $X$ be a radiant toric surface; here $n = 2$ and so $C=\{1, 2\}$.

If 1 and 2 are $\succeq$-incomparable, then $\R$ consists of two basic roots and $\Umax = \GG_a^2$ is the only unipotent group that acts faithfully with an open orbit on $X$. Equivalently, $X$ is a radiant toric variety of type~I.

Assume now that $1 \succeq 2$. Then
$$
\R_2^+ = \{ -q_2 \},\ \R_1^+ = \{ -q_1, \ -q_1 + q_2, \ \dots, \ -q_1 + d q_2 \}
$$
for some $d \geq 1$. By definition of a Demazure root, the number $d$ is determined by the following conditions on elements of the ray matrix $A$:
$$
a_{k1} \ge d a_{k2} \quad \text{for} \ \text{all} \ 3 \le k \le m \quad \text{and} \quad a_{k1} < (d + 1) a_{k2} \quad \text{for} \ \text{some} \ 3 \le k \le m.
$$
In this case the group $\Umax$ is metabelian of nilpotency class $d + 1$ and
$$
\Umax \cong \GG_a \ltimes \GG_a^{d + 1}.
$$

Applying Algorithm~\ref{alg} to this setting, we see that there are $d+1$ non-isomorphic regular unipotent subgroups that act faithfully with an open orbit on $X$:
$$
U_{-q_2} \ltimes \big( U_{-q_1} \times U_{-q_1 + q_2} \times \cdots \times U_{-q_1 + l q_2} \big) \cong \GG_a \ltimes \GG_a^{l + 1}, \quad 0 \le l \le d.
$$
By Proposition~\ref{pr:Nor}.(2) the action on the normalized subgroup is defined by
$$
u_{-q_2} (\alpha') u_{-q_1 + k q_2} (\alpha) u_{-q_2} (-\alpha') =
$$
$$
= \prod_{i=0}^k u_{-q_1 + (k-i) q_2} \left( \binom{k}{i} \alpha (\alpha')^i \right) = \prod_{j=0}^k u_{-q_1 + j q_2} \left( \binom{k}{j} \alpha (\alpha')^{k-j} \right).
$$

\smallskip

Now let us take a closer look at smooth toric surfaces. Such a surface is given by primitive vectors $p_0, p_1, \ldots, p_m, p_{m+1}$ with $p_0 = p_m$ and $p_1 = p_{m+1}$ such that any pair $p_s, p_{s + 1}$ forms a basis of the lattice $N = \ZZ^2$. In this case $p_{s - 1} + p_{s + 1} = c_s p_s$ for some integer numbers~$c_s$, ${1\le s\le m}$; see~\cite[Section~1.7]{Oda} or~\cite[Section~2.5]{F}. Clearly, the sequence $(c_1, \ldots, c_m)$ determines a smooth complete toric surface uniquely.

\begin{Proposition} \label{propseq}
Let $X$ be a smooth complete toric surface given by a sequence $(c_1, \ldots, c_m)$. Then $X$ is radiant if and only if the sequence $(c_1, \ldots, c_m, c_1)$ contains at least two standing nearby non-positive numbers.
\end{Proposition}

\begin{proof}
A complete toric surface $X$ is radiant if and only if for some $s$ all the vectors $p_i$, $i \ne s, s+1$ are contained in the negative orthant with respect to the basis $p_s, p_{s + 1}$. We have
$$
p_{s - 1}=c_s p_s - p_{s + 1}, \quad p_{s + 2} = c_{s + 1} p_{s + 1} - p_s,
$$
and all other vectors $p_i$ lie between $p_{s - 1}$ and $p_{s + 2}$. This implies the claim.
\end{proof}

For $m = 3$ we have only the projective plane $\PP^2$ represented by the sequence $(-1, -1, -1)$. The case $m = 4$ corresponds to Hirzebruch surfaces $\FF_q$ represented by sequences $(0, q ,0, -q)$ for some non-negative integer $q$. Since any smooth complete toric surface can be obtained either from $\PP^2$ or from $\FF_q$ by a sequence of blow-ups at $T$-fixed points, any sequence $(c_1, \ldots, c_m)$ can be obtained either from $(-1, -1, -1)$ or from $(0, q, 0, -q)$ by a series of operations, which add $1$ to some elements $c_s$ and $c_{s+1}$ (we assume that $c_{m + 1} = c_1$) and insert $1$ between them; see~\cite[Section~1.7]{Oda} or~\cite[Section~2.5]{F} for details.

Proposition~\ref{propseq} and the description of sequences $(c_1, \ldots, c_m)$ given above imply that for $m \le 5$ all smooth complete toric surfaces are radiant. For $m = 6$ we have the blow-up of $\PP^2$ at three $T$-fixed points corresponding to the sequence $(1, 1, 1, 1, 1, 1)$, which is not radiant.

\smallskip 

It is an interesting problem to study smooth radiant toric varieties in higher dimensions.

%%%%%%%%%%%%%%%%%%%%%%%%%%%%%%%%%%%%%%%%

{}
\end{document}